\documentclass[11pt, twoside, leqno, a4paper]{amsart}  
\usepackage{lipsum}
\usepackage{amsfonts}
\usepackage{graphicx}
\usepackage{epstopdf}
\usepackage{algorithmic}
\usepackage{calligra}
\usepackage{amsfonts,amsmath,amsthm,amssymb}
\usepackage{mathtools}
\usepackage{hyperref}
\usepackage{autonum}
\usepackage{hhline}
\usepackage{array}
\usepackage{diagbox}
\usepackage{tcolorbox}
\usepackage{mdframed}
\usepackage{multicol}
\usepackage{graphicx}
\usepackage{subcaption}
\usepackage{moreverb}
\usepackage{bbm}
\usepackage[margin=1.4in]{geometry}
\usepackage{todonotes}
\usepackage{scalerel,amssymb}
\allowdisplaybreaks
\usepackage{mathrsfs}  
\usepackage{lineno}
\usepackage{todonotes}
\usepackage{tikz}
\usepackage{pgfplots}
\pgfplotsset{compat=1.7}
\usetikzlibrary{arrows.meta}
\usepackage[numbers,sort&compress]{natbib}

\newcommand{\dt}{\, \textup{d} t}
\newcommand{\ds}{\, \textup{d} s }
\newcommand{\dx}{\, \textup{d} x}

\newcommand{\dint}{\displaystyle\int}

\newcommand{\R}{\mathbb{R}} 

\newcounter{rownumber}

\newtheorem{theorem}{Theorem}
\newtheorem{lemma}{Lemma}
\newtheorem{proposition}{Proposition}
\newtheorem*{assumption*}{Assumptions on the memory kernel}

\numberwithin{lemma}{section}
\numberwithin{proposition}{section}
\numberwithin{theorem}{section}
\numberwithin{equation}{section}
\makeatletter
\newcommand{\leqnomode}{\tagsleft@true}
\newcommand{\reqnomode}{\tagsleft@false}
\makeatother   

\title[The JMGT wave equation in hereditary fluids]{On the Jordan--Moore--Gibson--Thompson wave equation in hereditary fluids with quadratic gradient nonlinearity}


\author[V. Nikoli\'c \& B. Said-Houari ]{\bfseries Vanja Nikoli\'c$^*$ and Belkacem Said-Houari}

\address{ 
Department of Mathematics \\ 
Radboud University   \\ 
Heyendaalseweg 135,
6525 AJ Nijmegen, The Netherlands}
\email{vanja.nikolic@ru.nl} 

\address{  
	Department of Mathematics\\ College of Sciences\\ University of
Sharjah, P. O. Box: 27272 \\ Sharjah, United Arab Emirates}
\email{bhouari@sharjah.ac.ae}
\thanks{$^*$Corresponding author: Vanja Nikoli\'c, \href{mailto:vanja.nikolic@ru.nl}{vanja.nikolic@ru.nl}}
\begin{document}
\vspace*{12mm}
\begin{abstract}
We prove global solvability of the third-order in time Jordan--More--Gibson--Thompson acoustic wave equation with memory in $\R^n$, where $n \geq 3$. This wave equation models ultrasonic propagation in relaxing hereditary fluids and incorporates both local and cumulative nonlinear effects. The proof of global solvability is based on a sequence of high-order energy bounds that are uniform in time, and derived under the assumption of an exponentially decaying memory kernel and sufficiently small and regular initial data.
\end{abstract}  
\vspace*{10mm} 
\maketitle           
    
\section{Introduction}
The present paper focuses on the analysis of a third-order in time integro-differential equation arising in nonlinear acoustic wave propagation through viscous fluids with memory. We are motivated in this study by an increasing number of applications of high-frequency sound waves in medicine and industry~\cite{maresca2017nonlinear, he2018shared, melchor2019damage, duck2002nonlinear}. The propagation of ultrasound waves is naturally nonlinear and, therefore, the theory of nonlinear partial differential equations can offer a valuable insight into the ultrasonic wave behavior.  \\
\indent Of particular relevance here are the ultrasonic waves in relaxing hereditary media. The relaxation mechanisms can occur, for example, if there is an impurity in the fluid and they are known to introduce memory effects into propagation. The acoustic pressure then depends on the medium density at all previous times, resulting in a nonlocal in time wave equation; see, for example, the recent book~\cite{holm2019waves}. In particular, we investigate the following nonlinear acoustic wave equation:
\begin{equation}   
\begin{aligned}
&\tau \psi_{ttt}+\alpha \psi_{tt}-c^2\Delta \psi-b \Delta \psi_t+%
\displaystyle\int_{0}^{t}g(s)\Delta \psi(t-s)\ds
=\left( k \psi_{t}^{2}+
|\nabla \psi|^2\right)_t.
\end{aligned}
\end{equation}
We refer to Section~\ref{Sec:ProblemSetting}  below for its physical background. Our work here extends the analysis in~\cite{nikolic2020mathematical} by taking into account local effects in nonlinear sound propagation, which leads to a quadratic gradient nonlinearity in the model above. The energy arguments of ~\cite{nikolic2020mathematical} are then not adequate to capture such nonlinear effects. Instead, we have to involve higher-order energies of our system and devise new estimates. In addition, we extend the analysis in $\R^n$ to hold for all $n \geq 3$.  \\
\indent We organize the paper as follows. Section~\ref{Sec:ProblemSetting} discusses modeling of nonlinear acoustic propagation in relaxing hereditary media and the relevant related work. Section~\ref{Sec:Preliminaries} then collects theoretical results that are helpful in later proofs. In Section~\ref{Sec:FunctionalSetting}, we rewrite the problem as a first-order evolution equation and discuss the semigroup solution of the linearization. This serves as a basis in Section~\ref{Sec:LocalExistence} to prove the local solvability of the nonlinear problem. Section~\ref{Sec:EnergyEstimates} is devoted to deriving energy estimates that are uniform in time and thus crucial for proving global well-posedness. Finally, in Section~\ref{Sec:GlobalExistence}, we extend the existence result to $T=\infty$.
\section{Modeling and previous work  } \label{Sec:ProblemSetting}
Nonlinear acoustics studies sound waves of sufficiently large amplitudes, which makes using the full Navier--Stokes system of governing equations in fluid dynamics necessary. The equations connect the following quantities:
\begin{itemize}
	\item the pressure $u$, split into its mean and alternating part \[u=u_0+u' \quad \text{with} \ \nabla u_0=0;\] 
	\item the velocity $\boldsymbol{v}=\boldsymbol{v}_0+\boldsymbol{v}'$, which is assumed to be irrotational; \vspace*{1mm}
	\item the mass density $\varrho$, where $\varrho=\varrho_0+\varrho'$ with $\varrho_{0, t}=0$; \vspace*{1mm}
	\item the specific entropy $\eta$, where $\eta=\eta_0+\eta'$; \vspace*{1mm}
	\item the temperature $\theta$, where $\theta=\theta_0+\theta'$; \vspace*{1mm}
	\item the heat flux $\boldsymbol{q}$.
\end{itemize}
 The scalar field $u^\prime$ is called the acoustic pressure and vector field $\boldsymbol{v}^\prime$ the acoustic particle velocity. The governing equations include the conservation of momentum (the Navier--Stokes equation), mass, and energy: 
\begin{equation}
\begin{aligned}
& \varrho \frac{\partial \boldsymbol{v}^\prime}{\partial t}+\varrho (\boldsymbol{v}\cdot \nabla) \boldsymbol{v}+\nabla u^\prime=\left(\frac{4 \mu_v}{3}+\eta_v \right) \Delta \boldsymbol{v}^\prime,
 \\
&  \frac{\partial \varrho}{\partial t}+\nabla \cdot (\varrho \boldsymbol{v})=0, \\
& \varrho \left(\frac{\partial E}{\partial t}+(\nabla \cdot \boldsymbol{v})E \right)+u \nabla \cdot \boldsymbol{v} = K \Delta \theta ^\prime+\eta (\nabla \cdot \boldsymbol{v})^2+\frac{1}{2}(\partial_i v_j+\partial_j v_i-\frac{2}{3}\nabla \cdot \boldsymbol{v}\, \delta_{ij}),
\end{aligned}
\end{equation}
where $K$ is the heat conductivity, $\mu_v$ is the shear viscosity, $\eta_v$ denotes the bulk viscosity, and $\delta_{ij}$ is the Kroneker delta. The state equation relates the pressure and density within a fluid:
\begin{equation} 
u=u(\varrho, \eta).
\end{equation}
By expanding it in a Taylor series around the equilibrium state $(\varrho_0, \eta_0)$, one arrives at
\begin{equation}
\begin{aligned}
u=& \, \begin{multlined}[t]u_0+\varrho_0\Bigl(\frac{\partial u}{\partial \varrho}\Bigr)_{\varrho=\varrho_0, \eta}\frac{\varrho-\varrho_0}{\varrho_0}+\frac{1}{2}\varrho_0^2\Bigl(\frac{\partial^2 u}{\partial \varrho^2}\Bigr)_{\varrho=\varrho_0, \eta}\left(\frac{\varrho-\varrho_0}{\varrho_0}\right)^2\\+\Bigl(\frac{\partial u}{\partial \eta}\Bigr)_{\varrho, \eta=\eta_0}(\eta-\eta_0)+\ldots \end{multlined}
\end{aligned}
\end{equation}
This equation can be rewritten as
\begin{equation} 
\begin{aligned}
u-u_0=A\frac{\varrho-\varrho_0}{\varrho_0}+\frac{B}{2!}\Bigl(\frac{\varrho-\varrho_0}{\varrho}\Bigr)^2+\Bigl(\frac{\partial u}{\partial \eta}\Bigr)_{\varrho, \eta=\eta_0}(\eta-\eta_0)+\ldots,
\end{aligned}
\end{equation}
with the coefficients
\begin{equation}
A=\varrho_0\Bigl(\frac{\partial u}{\partial \varrho}\Bigr)_{\varrho=\varrho_0, \eta} \equiv \varrho_0 c_0^2, \qquad B=\varrho_0^2\Bigl(\frac{\partial^2 u}{\partial \varrho^2}\Bigr)_{\varrho=\varrho_0, \eta},
\end{equation}
where $c>0$ is the speed of sound. Hence we have 
\begin{equation} \label{Taylor_expansion}
\begin{aligned}
u-u_0=c^2(\varrho-\varrho_0)+\frac{c^2}{\varrho_0}\frac{B}{2A}(\varrho-\varrho_0)^2+\Bigl(\frac{\partial u}{\partial \eta}\Bigr)_{\varrho, \eta=\eta_0}(\eta-\eta_0)+\ldots
\end{aligned}
\end{equation}
The ratio $B/A$ indicates the nonlinearity of the equation of state for a given fluid. Lastly, to close the system, the classical Fourier law of heat conduction is employed in the equation for the conservation of energy:
\begin{equation}
\boldsymbol{q}= -K \nabla \theta,
\end{equation} 
\indent The so-called weakly nonlinear acoustic modeling introduces and studies approximations of these governing equations. We refer to~\cite{jordan2016survey} for a survey on such models.  Since the fluid is irrotational, a scalar acoustic velocity potential can be introduced as
\begin{equation}
\begin{aligned}
& \boldsymbol{v}= -\nabla \psi, \\
& u-u_0=-\varrho \psi_t,
\end{aligned}
\end{equation}
and these weakly nonlinear acoustic equations are typically expressed in terms of $\psi$.  The deviations of $\varrho$, $u$, $\eta$, and $\theta$ from their equilibrium values are assumed to be small. By neglecting all third and higher order terms in the deviations when combining the governing equations, one obtains the classical Kuznetsov wave equation:
\begin{equation}\label{Kuznt}
\psi_{tt}-c^{2}\Delta \psi-\delta \Delta \psi_{t}=\left( \frac{1}{c^{2}}\frac{B}{2A}(\psi_{t})^{2}+|\nabla \psi|^{2}\right)_t. 
\end{equation}
For a detailed derivation, we refer the reader to~\cite{crighton1979model, coulouvrat1992equations, kuznetsov1971equations, kaltenbacher2007numerical, kaltenbacher2018fundamental}. The coefficient $\delta>0$ in the model denotes the so-called sound diffusivity. This strongly damped nonlinear equation is also by now well-understood from the point of view of a rigorous mathematical analysis; see, for example,~\cite{mizohata1993global, DekkersRozanova}. In~\cite{kaltenbacher2012analysis, Wilke}, the well-posedness analysis is performed for the equation reformulated in terms of the pressure and acoustic particle velocity.\\
\indent The Kuznetsov equation incorporates both local and cumulative nonlinear effects. If local effects can be neglected, then
\begin{equation} \label{local_approximation}
|\nabla \psi|^2 \approx \frac{1}{c^2}\psi_t^2,
\end{equation} 
and one obtains the simpler Westervelt model~\cite{westervelt1963parametric}. \\
\indent  It is known, however, that the Fourier law used in the derivation of the Kuznetsov equation predicts an infinite speed of heat propagation~\cite{liu1979instantaneous}. Indeed, the strong $\delta$ damping in the model is proven to lead to a parabolic-like behavior with anexponential decay of energy~\cite{kaltenbacher2012analysis, Wilke}. The Maxwell--Cattaneo temperature law can be used instead to avoid this paradox for acoustic waves:
\begin{equation}
\tau \boldsymbol{q}_t+\boldsymbol{q}=-K \nabla \theta,
\end{equation}
where $\tau>0$ denotes the thermal relaxation time. In this manner, one obtains a third-order hyperbolic equation: 
\begin{equation}\label{JMGT}
\tau \psi_{ttt}+\psi_{tt}-c^{2}\Delta \psi-b\Delta \psi_{t}= \left( \frac{1}{c^{2}}\frac{B}{2A}(\psi_{t})^{2}+|\nabla \psi|^2\right)_t,
\end{equation}
often called the  Jordan--Moore--Gibson--Thompson (JMGT) equation (of Kuznetsov-type); see~\cite{jordan2008nonlinear}. The coefficient $b>0$ in the equation is computed as 
\begin{equation} \label{b}
b=\delta +\tau c^2.
\end{equation}
This model is mathematically well-studied in terms of well-posedness and regularity of solutions both in bounded domains~\cite{KaltenbacherNikolic} and in $\R^3$~\cite{Racke_Said_2019}. Furthermore, in~\cite{KaltenbacherNikolic}, it is rigorously justified that the limit of \eqref{JMGT}  as $\tau \rightarrow 0$ leads to the Kuznetsov equation.  By neglecting local nonlinear effects and assuming \eqref{local_approximation}, a Westervelt-type JMGT equation is obtained. For its analysis in smooth bounded domains, we refer to~\cite{kaltenbacher2012well, KaltenbacherNikolic}. \\
\indent The basis of the nonlinear analysis lies in the results for the linearization, often referred to as the Moore--Gibson--Thompson equation:
\begin{equation}\label{MGT_2}
\tau \psi_{ttt}+\alpha\psi_{tt}-c^{2}\Delta \psi-b \Delta \psi_{t}=0,
\end{equation}
\noindent where $\alpha>0$ stands for the friction. When $b>0$, the results of~\cite{Kaltenbacher_2011} show that the linear dynamics of this model is described by a strongly continuous semigroup. The semigroup is exponentially stable provided that the so-called \emph{subcritical} condition holds:
\begin{eqnarray}\label{Condition_Parameters}
\alpha b-\tau c^2>0. 
\end{eqnarray}
Interestingly, the semigroup is conservative in the critical case $\alpha b=\tau c^2$ -- the energy is conserved despite the presence of sound diffusivity.  The linear MGT equation has been extensively studied lately; see, for example,~\cite{marchand2012abstract, conejero2015chaotic, pellicer2019optimal, PellSaid_2019_1, bucci2019feedback, bucci2019regularity}.\\
\indent  When relaxation processes occur in high-frequency waves, the acoustic pressure can depend on the medium density at \emph{all prior} times. These processes happen, for example, when there is an impurity in the fluid. The pressure-density relation then involves a memory term:
\begin{equation}
u-u_0=c^2(\varrho-\varrho_0)+\frac{c^2}{\varrho_0}\frac{B}{2A}(\varrho-\varrho_0)^2+\Bigl(\frac{\partial u}{\partial \eta}\Bigr)_{\varrho, \eta=\eta_0}(\eta-\eta_0)-\int_0^t g(t-s) \varrho'(s)\, \textup{d} s;
\end{equation}
cf.~\cite{holm2019waves}. The function $g$ is the relaxation memory kernel related to the occurring relaxation mechanism. \\
\indent Such memory effects in nonlinear wave propagation motivate our work in the present paper. Our object of study is the non-local JMGT equation with quadratic gradient nonlinearity given by
\begin{equation} \label{Main_Equation} 
\begin{aligned}
&\tau \psi_{ttt}+\alpha \psi_{tt}-c^2\Delta \psi-b \Delta \psi_t+%
\displaystyle\int_{0}^{t}g(s)\Delta \psi(t-s)\ds
=\left( k \psi_{t}^{2}+
|\nabla \psi|^2\right)_t.
\end{aligned}
\end{equation}
The constant $k \in \R$ indicates the nonlinearity of the equation. In the present work, we consider memory that involves only the acoustic velocity potential and not its time derivatives, which is sometimes regarded as memory of type I. \\
\indent The linear model associated with the JMGT equation with memory in the pressue form
\begin{equation}   \label{Main_Equation_linear}
\begin{aligned}
&\tau u'_{ttt}+\alpha u'_{tt}-c^2\Delta u'-b \Delta u'_t+%
\displaystyle\int_{0}^{t}g(s)\Delta z(t-s)\ds
=0.
\end{aligned}
\end{equation}
 has also been subject to extensive study recently. In~\cite{lasiecka2016moore}, the effects of different memories on the stability of the equation are investigated. In particular, the authors analyzed the equation in the pressure form with $z=u'$ (memory of type I), $z=u'_{t}$ (type II), or  $z= u'+u'_{t}$ (type III). If the memory kernel $g$ decays exponentially, the same holds for the solution for all three types of memory, provided that the non-critical condition \eqref{Condition_Parameters} holds. This result is extended in~\cite{lasiecka2015moore} to include not only an exponential decay rate of the memory kernel. \\
\indent The critical case $\alpha b =\tau c^2$, where the presence of memory is essential, is analyzed in~\cite{dell2016moore} with memory of type I (depending only on the pressure). With a strictly positive self-adjoint linear operator $A$ in place of $-\Delta$, the problem is exponentially stable if and only if $A$ is a bounded operator. In the case of an unbounded operator $A$, the corresponding energy decays polynomially with the rate $1/t$ for regular initial data.\\
\indent Neglecting local nonlinear effects in \eqref{Main_Equation} via relation \eqref{local_approximation} leads to a Westervelt-type JMGT equation with memory. In ~\cite{lasiecka2017global}, this third-order nonlinear equation is studied in the pressure form in the critical case $\alpha b= \tau c^2$ in the presence of memory type III (depending on $u'$ and $u'_t$). It is shown that an appropriate adjustment of the
memory kernel allows having solutions globally in time for sufficiently small and regular initial data. Global solvability for the equation in potential form in $\R^3$ for small and regular data is proven in~\cite{nikolic2020mathematical} in the non-critical case and with memory depending only on $\psi$. \\
\indent In relaxing media, the memory kernel is often given by the exponential function
\begin{equation}
g(s)=m c^2\exp{(-s/\tau)},
\end{equation}
where $m$ is the relaxation parameter; see~\cite[Chapter 1]{naugolnykh2000nonlinear} and \cite[Section 1]{lasiecka2017global}. Motivated by this, we make the following assumptions on the relaxation kernel throughout the paper; see also~\cite{dell2016moore, nikolic2020mathematical}.
\begin{assumption*} The memory kernel is assumed to satisfy the following conditions:
\begin{enumerate}
	\item[(G1)] \label{itm:first}  $g\in W^{1,1}(\R^+)$ and $g'$ is almost continuous on $\R^+=(0, +\infty)$. \vspace{0.1 cm}
	\item[(G2)] $g(s) \geq 0$ for all $s>0$ and 
	\reqnomode
	\begin{align} 
0< \int_{0}^{\infty}g(s)\ds < c^2. 	\vspace{0.1 cm}
	\end{align}
	\item[(G3)] There exists $\zeta>0$, such that the function $g$ satisfies the 
	differential inequality given by
	\begin{equation}
	g^\prime(s)\leq -\zeta g(s)
	\end{equation}
	 for every $s\in (0,\infty)$. 	\vspace{0.1 cm}
	\item[(G4)] It holds that $g^{\prime\prime}\geq0$ almost everywhere.\vspace{0.2 cm}
\end{enumerate}
\end{assumption*}
\section{Auxiliary theoretical results}\label{Sec:Preliminaries}
For future use in the analysis of equation \eqref{Main_Equation}, we recall here several helpful theoretical results and set the notation. We choose to work in the history framework of Dafermos~\cite{dafermos1970asymptotic}, following previous work on acoustic equations with memory; see, for example,~\cite{dell2016moore}. This is achieved by introducing the auxiliary past-history variable $\eta=\eta(t,s)$ for $t \geq 0$, defined as
\begin{equation}  \label{def of eta}
\eta(s)= \begin{cases}
\psi(t)-\psi(t-s), \quad  &0<s\leq t, \\
\psi(t), \quad &s>t.
\end{cases}
\end{equation}
If we choose $\eta\vert_{t=0}=\psi_0(x)$, we can rewrite our problem as 
\begin{equation}  \label{eta syst}
\begin{cases}
\tau \psi_{ttt}+\alpha \psi_{tt}-b \Delta \psi_{t}-c^2_g\Delta \psi-%
\displaystyle\int_{0}^{\infty}g(s)\Delta \eta(s)\ds 
= 2k\psi_{t}\psi_{tt}+2 \nabla \psi \cdot \nabla \psi_t,  \\[2mm] 
\eta_{t}(x,s)+\eta_{s}(x,s)=\psi_{t}(x,t),%
\end{cases}%
\end{equation}
where we have introduced the modified speed
\begin{align} \label{def_cg}
\ c^2_g=c^2-\displaystyle\int_{0}^{\infty}g(s)\ds. 	
\end{align}
Thanks to our assumptions on the memory kernel, we have $c^2_g>0$. The equations are supplemented with the initial data
\begin{equation}  \label{initial data}
\psi(x,0)=\psi_{0}(x),\qquad \psi_{t}(x,0)=\psi_{1}(x), \qquad \psi_{tt}(x,0)=\psi_{2}(x).
\end{equation}
\paragraph{\bf Setting $\boldsymbol{\alpha=1}$.}  We can take $\alpha=1$  without the loss of generality because we can always re-scale other coefficients in the equation. The subcritical condition  then reads as 
\begin{equation}\label{critical_condition_scaled}
b >\tau c^2,
\end{equation}
which is equivalent to requiring the sound diffusivity $\delta$ to be positive since $b=\delta+ \tau c^2$. The critical case corresponds to $b=\tau c^2$. We always require the presence of memory via the assumption $\tau c^2 > \tau c^2_g$; that is, $\int_0^\infty g(s)\ds >0$.
\subsection{Notation} In the present work, the constant $C$ always stands for a generic positive constant that does not depend on time, and it may have different value on different occasions. We write $x \lesssim y$ instead of $x \leq C y$.
\subsection{Helpful inequalities} We collect here several inequalities that are used throughout the proofs. We often rely on the following endpoint Sobolev embedding:
\begin{equation} \label{GNS_ineq}
\|\psi\|_{L^{\frac{2n}{n-2}}} \lesssim \|\nabla \psi\|_{L^2},
\end{equation}
which holds provided that $n \geq 3$; cf.~\cite{bahouri2011fourier}. Going forward, we therefore  assume $n \geq 3$. To treat the tri-linear terms, we also employ the inequality
\begin{equation} \label{trilinear_est}
\|f g h\|_{L^1} \lesssim \|f\|_{L^2}\|g\|_{L^n}\|h\|_{L^{\frac{2n}{n-2}}}
\end{equation}
together with the embedding
\begin{align} \label{embedding_2}
	&{H}^{\frac{n-2}{2}}(\R^n) \hookrightarrow L^n(\R^n). 
\end{align} 
Furthermore, when estimating the nonlinear terms in the equation, we will rely on the embedding
\begin{equation}
H^r(\R^n) \hookrightarrow L^{\infty}(\R^n), \qquad r> n/2.
\end{equation}
We will work with the space-differentiated equation as well, and so commutator estimates are particularly helpful. We introduce the commutator notation by
\begin{equation} \label{commutator}
[A,B]=AB-BA.
\end{equation}
It helps to note that
\begin{equation}\label{Derivative_Comuta}
\partial^\kappa(AB)=[\partial^\kappa,A]B+A\partial^\kappa B, \quad \kappa \geq 1.
\end{equation}
We have the following estimates.
\begin{lemma}[see Lemma 4.1 in~\cite{HKa06}.]
\label{Guass_symbol_lemma} Let $1\leq p,\,q,\,r\leq \infty $ and $%
1/p=1/q+1/r $. Then we have%
\begin{equation}
\Vert \nabla^\kappa( fg) \Vert _{L^p}\lesssim \Vert f\Vert _{L^q}\Vert \nabla^\kappa
g\Vert _{L^r}+\Vert g\Vert _{L^q}\Vert \nabla^\kappa f\Vert _{L^r} ,\quad \kappa \geq 0,
\label{First_inequaliy_Guass}
\end{equation}
and the commutator estimate
\begin{equation} \label{Second_inequality_Gauss}
\begin{aligned}
\Vert [ \nabla^{\kappa},f] g\Vert _{L^p}=&\Vert \nabla^\kappa(fg)-f\nabla^\kappa g\Vert_{L^p}\notag\\
\lesssim& \, \Vert \nabla f\Vert _{L^q}\Vert
\nabla^{\kappa-1}g\Vert _{L^r}+\Vert g\Vert _{L^q}\Vert \nabla^{\kappa}f\Vert _{L^r}
,\quad \kappa \geq 1.
\end{aligned}
\end{equation}
\end{lemma}
\noindent Finally, the following lemma will be useful in our energy arguments when proving global well-posedness.
\begin{lemma}[see Lemma 3.7 in~\cite{strauss1968decay}] 
\label{Lemma_Stauss} Let $M=M(t)$ be a non-negative continuous function such that
\begin{equation}
M(t)\leq C_1+C_2 M(t)^{\kappa}
\end{equation}
holds in some interval containing $0$, where $C_1$ and $C_2$ are positive constants and $\kappa>1$. If $M(0)\leq C_1$ and
\begin{equation}
C_1C_2^{1/(\kappa-1)}<(1-1/\kappa)\kappa^{-1/(\kappa-1)},
\end{equation}
then in the same interval it holds
\begin{equation}
M(t)<\frac{C_1}{1-1/\kappa}.
\end{equation}
\end{lemma}
\section{The first-order system and the linearization} \label{Sec:FunctionalSetting}
We can see our equation as a first-order in time system by introducing
\begin{equation}\label{Change_Variable}
v=\psi_{t}, \quad w=\psi_{tt}.
\end{equation}
The JMGT equation can then be restated as
\begin{equation}  \label{Main_System}
\begin{cases}
\psi_{t}=v, \\ 
v_{t}=w, \\ 
\tau w_{t}=- w+c^2_g\Delta \psi+b\Delta v + \displaystyle%
\int_{0}^{\infty}g(s)\Delta\eta(s)\ds+2(k vw+\nabla \psi \cdot \nabla v),
\\ 
\eta_{t}=v-\eta_{s},%
\end{cases}%
\end{equation}
with the initial data given by
\begin{equation}  \label{Main_System_IC}
(\psi, v, w, \eta) |_{t=0} =(\psi_0, \psi_1, \psi_2, \psi_0).
\end{equation}
We then adapt the functional framework of~\cite{dell2016moore} to our setting. To work in the past-history framework of Dafermos, we introduce the weighted $L^{2}$-spaces,
\begin{equation}
L^2_{\tilde{g}}=L^{2}_{\tilde{g}}(\mathbb{R}^{+}, L^2(\mathbb{R}^{n}))
\end{equation}
with three types of weights: $\tilde{g} \in \{g, -g' , g''\}$. The space is endowed with the inner product 
\begin{equation}
\left(\eta,\tilde{\eta} \right)_{L^2, \tilde{g}}=%
\displaystyle\int_{0}^{\infty}\tilde{g}(s)\left(
\eta(s),\tilde{\eta}(s)\right)_{L^{2}(\mathbb{R}%
	^{n})} \ds
\end{equation}
for $\eta, \tilde{\eta} \in L^2_{\tilde{g}}$. The corresponding norm is 
\begin{equation}
\Vert\eta\Vert^{2}_{L^2, \tilde{g}}=\int_{0}^{\infty}\tilde{g}(s)\Vert%
\eta(s)\Vert_{L^{2}}^{2}\ds,
\end{equation}
\indent Let $n\geq 3$. In order to formulate our results, for an integer $m \geq 1$, we introduce the Hilbert spaces
\begin{equation} \label{Hs}
\begin{aligned}
\mathcal{H}^{m-1}=&\, \begin{multlined}[t] \{ \psi: \ \nabla \psi, \, \ldots, \nabla^{(m)} \psi \in L^2(\R^n)\} \times H^{m}(\R^n) \times {H}^{m-1}(\R^n)
\times \mathcal{M}^m,
\end{multlined}
\end{aligned}
\end{equation}
where 
\begin{equation}
\mathcal{M}^m= \{\eta: \nabla \eta, \ldots, \nabla^{(m)} \eta\in L^2_{-g'}\}.
\end{equation}
The corresponding norm is given by
\begin{equation} \label{norm_Hm-1}
\begin{aligned}
\|\Psi\|^2_{\mathcal{H}^{m-1}}=\|\nabla \psi\|^2_{H^{m-1}}+\|v\|^2_{H^{m}}+\|w\|^2_{H^{m-1}}+\|\nabla \eta\|^2_{H^{m-1},-g'}.
\end{aligned}
\end{equation}
We recall that for $n \geq 3$, the endpoint Sobolev embedding \eqref{GNS_ineq} holds in $\R^n$, so $\|\cdot\|_{\mathcal{H}^{m-1}}$ is indeed a norm.\\
\indent To further arrive at an initial-value problem for a first-order evolution equation, we set $\Psi=(\psi, v, w, \eta)^T$
with $\Psi_0=\Psi(0)$. Moreover, we define the operator $ \mathcal{A}$ as
\begin{equation}
\begin{aligned}
\mathcal{A}\begin{bmatrix}
\psi \\[1mm]
v \\[1mm]
w \\[3mm]
\eta 
\end{bmatrix} = \begin{bmatrix}
v \\[1mm]
w \\[1mm]
-\dfrac{1}{\tau} w+ \frac{c^2_g}{\tau} \Delta \psi+\frac{b}{\tau} \Delta v+\frac{1}{\tau}\displaystyle \int_0^\infty g(s)\Delta \eta(s) \ds \\[3mm]
v+\mathbb{T} \eta
\end{bmatrix}
\end{aligned}
\end{equation}
with the domain
\begin{equation} \label{D(A)_m}
\begin{aligned}
D(\mathcal{A})
=\,  \left\{(\psi, v, w, \eta)^T \in \mathcal{H}^{m-1} \left\vert\rule{0cm}{1cm}\right. \begin{matrix}
w \in H^{m}(\R^n), \\[2mm]
\dfrac{c^2_g}{\tau} \Delta \psi+\dfrac{b}{\tau} \Delta v+\dfrac{1}{\tau}\displaystyle \int_0^\infty g(s)\Delta \eta(s)  \in H^{m-1}(\R^n), \\[4mm]
\eta \in D(\mathbb{T}) \end{matrix} \right\}.
\end{aligned}
\end{equation}
The linear operator $\mathbb{T}$ above is given by
\begin{equation} \label{def_T_eta}
\mathbb{T}\eta=-\eta_s,
\end{equation}
and has the domain
\begin{equation}
D(\mathbb{T})=\{\eta\in%
\mathcal{M}^m\, \big| \ \eta_s\in \mathcal{M}^m,\ \eta(s=0)=0\}.
\end{equation}
\noindent We can then formally see $\Psi$ as the solution to
\begin{equation} \label{abstract_evol_eq}
\begin{aligned}
\begin{cases}
\dfrac{\textup{d}}{\dt}\Psi(t)= \mathcal{A} \Psi(t)+\mathbb{F}(\Psi, \nabla \Psi),\quad t>0, \vspace{0.2cm}\\[1mm] 
\Psi(0)=\Psi_0,
\end{cases} 
\end{aligned}
\end{equation}
with the nonlinear term given by
\begin{equation} \label{def_F}
\begin{aligned}
\mathbb{F}(\Psi, \nabla \Psi)= \frac{1}{\tau} \, [
0,\
0 ,\
2kv w+ 2\nabla \psi \cdot \nabla v ,\
0 
]^T.
\end{aligned}
\end{equation}
In the spirit of~\cite{Racke_Said_2019}, our plan is to prove local solvability of \eqref{abstract_evol_eq} by introducing the mapping 
\begin{equation}
\begin{aligned}\label{Mapping_T}
\mathcal{T}(\Phi)= e^{t\mathcal{A}}\Psi_0+\int_0^te^{(t-r)\mathcal{A}}\mathbb{F}(\Phi, \nabla \Phi)(r)\, \textup{d}r
\end{aligned}
\end{equation}
on a suitable ball in a Banach space. If we can employ the Banach fixed-point on $\mathcal{T}$,  the unique fixed-point of this mapping is the mild solution of \eqref{abstract_evol_eq}. \\
\indent To begin with, we prove that $\mathcal{A}$ generates a linear $\textup{C}_0$-semigroup. The proof follows by adapting the arguments from \cite{Bounadja_Said_2019} based on the Lumer-Phillips theorem. To this end we first introduce an equivalent scalar product and norm on the space $\mathcal{H}^{m-1}$ adapted to fit our particular problem. \\
\indent For any vectors $\Psi=(\psi, v, w, \eta)$ and $\bar{\Psi}=(\bar{\psi}, \bar{v}, \bar{w}, \bar{\eta})$ in $\mathcal{H}^{m-1}$, we define the scalar product
\begin{equation} 
\begin{aligned}
&(\Psi, \bar{\Psi})_{\mathcal{H}^{m-1}}\\[1mm]
=& \,\begin{multlined}[t] \sum_{\kappa=0}^{m-1}\left \{ \vphantom{\int_{\mathbb{R}^{n}}} c^2_g
(\nabla^{\kappa +1}(\psi+\tau v), \nabla^{\kappa +1}(\bar{\psi}+\tau \bar{v}))_{L^2}+\tau(b -\tau
c^2_g )(\nabla^{\kappa+1} v, \nabla^{\kappa+1} \bar{v})_{L^2}  \right.\\[1mm] 
+\tau(b -\tau
c^2_g )(\nabla^{\kappa} v, \nabla^{\kappa} \bar{v})_{L^2}+(\nabla^{\kappa}(v+\tau w), \nabla^{\kappa}(\bar{v}+\tau \bar{w}))_{L^2}\\[2mm]+\tau (\nabla^{\kappa+1}
\eta, \nabla^{\kappa+1}
\bar{\eta})_{L^2, -g'}+ (\nabla^{\kappa+1} \eta, \nabla^{\kappa+1} \bar{\eta})_{L^2, g} 
\\[1mm] \left.
+ \tau \int_{\mathbb{R}^{n}} \left(
(\nabla^{\kappa+1} \eta, \nabla^{\kappa+1} \bar{v})_{L^2,g}+(\nabla^{\kappa+1} \bar{\eta}, \nabla^{\kappa+1} v)_{L^2,g}
 \vphantom{\int_{\mathbb{R}^{n}}} \right)\dx \right \}. \end{multlined}
\end{aligned}
\end{equation}
The corresponding norm is given by
\begin{equation} 
\begin{aligned}
|||\Psi|||_{\mathcal{H}^{m-1}}^2
=& \,\begin{multlined}[t] \sum_{\kappa=0}^{m-1} \left\{ \vphantom{\int_{\mathbb{R}^{n}}} c^2_g
\|\nabla^{\kappa +1}(\psi+\tau v)\|_{L^2}^{2}+\tau(b -\tau
c^2_g )\|\nabla^{\kappa+1} v\|^{2}_{L^2} \right. \\[1mm] 
+\tau(b -\tau
c^2_g )\|\nabla^{\kappa} v\|^{2}_{L^2}+\|\nabla^{\kappa}(v+\tau w)\|^{2}_{L^2}+\tau \Vert \nabla^{\kappa+1}
\eta \Vert _{L^2, -g'}^{2}
\\[1mm] \left. + \|\nabla^{\kappa+1} \eta\|^2_{L^2, g}
+ 2\tau \int_{\mathbb{R}^{n}}
(\nabla^{\kappa+1} \eta, \nabla^{\kappa+1} v)_{L^2, g}
 \dx  \right\}. \end{multlined} 
\end{aligned}
\end{equation}
We now prove that this problem-specific norm is equivalent to the standard norm \eqref{norm_Hm-1} on $\mathcal{H}^{m-1}$.
\begin{proposition} \label{Prop:Equiv}	Let $b \geq \tau c^2 > \tau c^2_g$ and $m \geq 1$. Assume that $n \geq 3$. There exist positive constants $C_1$ and $C_2$ such that
	\begin{equation} \label{equiv_norms}
		C_1\|\Psi\|_{\mathcal{H}^{m-1}} \leq 	|||\Psi|||_{\mathcal{H}^{m-1}} \leq 	C_2\|\Psi\|_{\mathcal{H}^{m-1}}
	\end{equation}
for all $\Psi \in \mathcal{H}^{m-1}$.	
\end{proposition}
\begin{proof}
The proof follows analogously to the proof of~\cite[Lemma 3.1]{dell2016moore}, so we only prove the more involved left-hand side inequality here.\\
\indent For each $\kappa=0, 1,\dots, m-1$, we have by Young's inequality  
\begin{equation}
\begin{aligned}
&\left|2\tau\int_{\R^n}\int_{0}^\infty g(s)\nabla^{\kappa+1} \eta(s) \cdot \nabla^{\kappa+1} v \ds\dx \right| \\
\leq& \, \frac{\tau^{2}(c^2-c^2_g)}{\varepsilon+1}\Vert\nabla^{\kappa+1} v\Vert^{2}_{L^2}+(\varepsilon+1)\dint_{0}^\infty g(s)\Vert\nabla^{\kappa+1}\eta(s)\Vert^{2}_{L^2}\ds,
\end{aligned}
\end{equation}
for all $\varepsilon>0$. We recall assumption (G3) on the relaxation kernel $g$ to derive
\begin{equation}
\begin{aligned}
&2\tau\int_{\R^n}\int_{0}^\infty g(s)\nabla^{\kappa+1} \eta(s) \cdot \nabla^{\kappa+1} v \ds\dx \\
\geq& \, -\frac{\tau^{2}(c^2-c^2_g)}{\varepsilon+1}\Vert\nabla^{\kappa+1} v\Vert^{2}_{L^2}- \|\nabla^{\kappa+1} \eta\|^2_{L^2, g} 
-\frac{\varepsilon}{\zeta}\dint_{0}^\infty(- g'(s))\Vert\nabla^{\kappa+1}\eta(s)\Vert^{2}_{L^2}\ds.
\end{aligned}
\end{equation}
We know that $c^2 > c^2_g$ and we can take $\varepsilon< \tau \zeta$ to obtain
\begin{equation} \label{left-hand side}  
\begin{aligned}
&\begin{multlined}[t]  c^2_g
\|\nabla^{\kappa +1}(\psi+\tau v)\|_{L^2}^{2}+\tau(b -\tau
c^2_g )\|\nabla^{\kappa+1} v\|^{2}_{L^2} +\tau(b -\tau
c^2_g )\|\nabla^{\kappa} v\|^{2}_{L^2}
\\[1mm]+\|\nabla^{\kappa}(v+\tau w)\|^{2}_{L^2}  +\tau \Vert \nabla^{\kappa+1}
\eta \Vert _{L^2, -g'}^{2}+ \|\nabla^{\kappa+1} \eta\|^2_{L^2, g} \\[1mm]
+ 2\tau \int_{\mathbb{R}^{n}}
(\nabla^{\kappa+1} \eta, \nabla^{\kappa+1} v)_{L^2, g}
\dx   
\end{multlined}\\
\geq&\, \begin{multlined}[t] c^2_g\Vert\nabla^{\kappa+1}(\psi +\tau v)\Vert^{2}_{L^2}  +\left[\tau (b-\tau c^2_g)-\tau^2(c^2-c^2_g)/(\varepsilon+1) \right] \Vert \nabla^{\kappa+1} v\Vert^{2}_{L^2} \\[1mm]
+\Vert \nabla ^{\kappa}(v+\tau w)\Vert^{2}_{L^2}+(\tau-\varepsilon/\zeta)\Vert \nabla^{\kappa+1}\eta\Vert^{2}_{L^2, -g'}. 
\end{multlined}
\end{aligned}
\end{equation}
Note that the condition 
\begin{equation}
\tau (b-\tau c^2_g)-\frac{\tau^2(c^2-c^2_g)}{\varepsilon+1}>0
\end{equation}
is equivalent to $\varepsilon \tau (b-\tau c^2_g)+\tau (b- \tau c^2)>0$, which holds true under the assumptions of the theorem. It remains to observe that
\begin{equation} \label{aux_ineq}
\begin{aligned}
&\begin{multlined}[t]c^2_g\Vert\nabla^{\kappa+1}(\psi +\tau v)\Vert^{2}_{L^2}+\tau (b-\tau c^2_g)(\Vert\nabla^{\kappa+1} v\Vert^{2}_{L^2}+\|\nabla^\kappa v\|_{L^2})\\ +\Vert \nabla ^{\kappa}(v+\tau w)\Vert^{2}_{L^2}
\end{multlined}
\\
\geq&\, \tilde{C}(\Vert\nabla^{\kappa+1}\psi \Vert^{2}_{L^2}+\Vert\nabla^{\kappa+1}v \Vert^{2}_{L^2} +\Vert \nabla ^{\kappa}v\Vert^{2}_{L^2}+\Vert \nabla ^{\kappa}w\Vert^{2}_{L^2}),
\end{aligned}
\end{equation}
for some positive constant $\tilde{C}$ that depends on $\tau$, $c^2_g$, and $b$. The left-hand side inequality in \eqref{equiv_norms} then follows by summing \eqref{left-hand side}  over $\kappa=0, 1,\dots, m-2$ and using estimate \eqref{aux_ineq}. 
\end{proof}	
\begin{theorem}
	Let $b \geq \tau c^2 > \tau c^2_g$ and $m \geq 1$. Assume that $n \geq 3$. Then the linear operator $\mathcal{A}$
	is the infinitesimal generator of a linear $\textup{C}_0$-semigroup 
	\begin{equation}
	S(t) = e^{t\mathcal{A}}: \, \mathcal{H}^{m-1} \rightarrow \mathcal{H}^{m-1}.
	\end{equation}
\end{theorem}
\begin{proof}
The proof is based on the Lumer-Phillips theorem applied to a bounded perturbation of the operator $\mathcal{A}$; see~\cite[Chapter 1.4]{pazy2012semigroups}. To this end, we introduce the operator $B$ by
\begin{equation}
\begin{aligned}
B\begin{bmatrix}
\psi \\[1mm]
v \\[1mm]
w \\[3mm]
\eta 
\end{bmatrix} = \begin{bmatrix}
0 \\[1mm]
0 \\[1mm]
-\frac{1}{\tau}(b-\tau c_g^2)v \\[3mm]
0
\end{bmatrix}
\end{aligned}
\end{equation}
and set $\mathcal{A}_{\mathcal{B}}=\mathcal{A}+B$; see~\cite[Theorem 2.4]{Bounadja_Said_2019} for a similar approach. Since $B$ is a bounded linear operator on $\mathcal{H}^{m-1}$, if $\mathcal{A}_{B}$ generates a $C_0$ semigroup on $\mathcal{H}^{m-1}$, then so does $\mathcal{A}=\mathcal{A}_{\mathcal{B}}-B$; cf.~\cite[Theorem 1.1]{pazy2012semigroups}. \\
\indent We first wish to prove that $\mathcal{A}_{B}$ is dissipative:
\[(\mathcal{A}_{B}\Psi, \Psi)_{\mathcal{H}^{m-1}} \leq 0 \quad \text{for all} \ \Psi \in D(\mathcal{A}).\]
For a given $\Psi \in D(\mathcal{A})$, it is straightforward to check that the following identity holds:
\begin{equation}
\begin{aligned}
(\mathcal{A}_{\mathcal{B}}\Psi, \Psi)_{\mathcal{H}^{m-1}}=&\, \sum_{\kappa=0}^{m-1}\begin{multlined}[t]\left(\vphantom{\int_0^\infty}
-(b-\tau c^2)\|\nabla^{\kappa+1} v\|^2_{L^2}-\frac{1}{\tau}(b-\tau c^2_g)\|\nabla^{\kappa} v\|^2_{L^2} \right.\\
+\int_0^\infty(g-\tau g')\nabla^{\kappa+1} \eta(s) \cdot \nabla^{\kappa+1} \mathbb{T}\eta(s)\ds \\
+\int_0^\infty g(s) \nabla^{\kappa+1} \mathbb{T}\eta(s)\cdot \nabla^{\kappa+1} v \ds \\ \left.
-\int_0^\infty g'(s)\nabla^{\kappa+1} \eta(s)\cdot \nabla^{\kappa+1} v\ds \right).
 \end{multlined}
\end{aligned}
\end{equation}
We can simplify this expression by noting that
\[\int_0^\infty g(s) \nabla^{\kappa+1} \mathbb{T}\eta(s)\cdot \nabla^{\kappa+1} v \ds = \int_0^\infty g'(s)\nabla^{\kappa+1} \eta(s)\cdot \nabla^{\kappa+1} v\ds.\]
Therefore, we have
\begin{equation}
\begin{aligned}
(\mathcal{A}_{B}\Psi, \Psi)_{\mathcal{H}^{m-1}}=&\, \sum_{\kappa=0}^{m-1}\begin{multlined}[t]\left(\vphantom{\int_0^\infty}
-(b-\tau c^2)\|\nabla^{\kappa+1} v\|^2_{L^2}-(b-\tau c^2_g)\|\nabla^{\kappa} v\|^2_{L^2}\right.\\ \left.
+\int_0^\infty(g'-\tau g'')\|\nabla^{\kappa+1} \eta(s)\|_{L^2}^2\ds \right) \leq 0,
\end{multlined}
\end{aligned}
\end{equation}
where the last inequality holds thanks to the assumptions on the memory kernel. To be able to employ the Lumer-Phillips theorem, it remains to prove that
\begin{equation}
\textup{Ran}(\textup{I}-\mathcal{A}_{B})=\mathcal{H}^{m-1}.
\end{equation}
In other words, for a given $F=(f, g, h, p) \in \mathcal{H}^{m-1}$, we have to prove existence of a unique solution $\Psi=(\psi, v, w, \eta) \in D(\mathcal{A})$ to the equation
\begin{equation}
\Psi- \mathcal{A}_{B}\Psi=F.
\end{equation}
We can write this equation component-wise as the following system:
\begin{equation}  \label{aux_System}
\begin{cases}
\psi=v+f, \\ 
v=w+g, \\ 
 w=- \frac{1}{\tau}w+\frac{1}{\tau}c^2_g\Delta \psi+\frac{1}{\tau}b\Delta v + \frac{1}{\tau}\displaystyle%
\int_{0}^{\infty}g(s)\Delta\eta(s)\ds-\tfrac{1}{\tau}(b-\tau c^2_g)v+h,
\\ 
\eta=v-\eta_{s}+p.%
\end{cases}%
\end{equation}
Solving the last equation and using $\eta(s=0)=0$ yields
\begin{equation} \label{eq_eta}
\eta(s)=(1-e^{-s})v+\int_0^s e^{-(s-y)}p(y)\, \textup{d}y.
\end{equation}
Combining the equations in the system results in an elliptic problem for $v$:
\begin{equation}  \label{eq_v}
- \nu \Delta v +\sigma v =q,
\end{equation}
where the coefficients and the source term are given by
\begin{equation}
\begin{aligned}
&\nu = b+c^2_g+\int_0^\infty g(s)(1-e^{-s})\ds>0,\\
& \sigma = 1+\tau + (b-\tau c^2_g), \\
& q= c^2_g \Delta f+(1+ \tau)g+\tau h +\int_0^\infty g(s)\int_0^s e^{-(s-y)}\Delta p(y)\,\textup{d}y \textup{d}s.
\end{aligned}
\end{equation}
To prove that the elliptic equation \eqref{eq_v} admits a unique solution $v \in H^{m}(\R^n)$, we note that 
\begin{equation}
\|q\|_{H^{m-2}} \lesssim \begin{multlined}[t]\|\Delta f\|_{H^{m-2}}+\|g\|_{H^{m-2}}+\|h\|_{H^{m-2}}\\
+\int_0^\infty g(s)\int_0^s e^{-(s-y)}\|\Delta p(y) \|_{H^{m-2}}\,\textup{d}y\ds. \end{multlined}
\end{equation}
We can further estimate the last term on the right as follows:
\begin{equation} \label{est_g}
\begin{aligned}
&\int_0^\infty g(s)\int_0^s e^{-(s-y)}\|\Delta p(y) \|_{H^{m-2}}\,\textup{d}y\ds\\
\leq& \, \sqrt{c^2-c^2_g} \left(\int_0^\infty g(s)\left(\int_0^s e^{-(s-y)}\|\Delta p(y) \|_{H^{m-2}}\,\textup{d}y\right)^2\ds \right)^{1/2}\\
\leq& \sqrt{\frac{c^2-c^2_g}{\zeta}} \left(\int_0^\infty -g'(s) \|\Delta p(y) \|^2_{H^{m-2}}\ds \right)^{1/2}.\,
\end{aligned}
\end{equation}
The claim then follows by the Lax-Milgram theorem. From \eqref{aux_System}, we further have $\psi=v+f \in \{\psi: \nabla \psi,\ldots, \nabla^{m} \psi \in L^2(\R^n)\}$ and $w=v-g \in H^{m}(\R^n)$. On account of equation \eqref{eq_v}, it follows that
\begin{equation}
\begin{aligned}
\|\nabla \eta\|_{H^{m-1}, -g'}\lesssim&\, \|\nabla v\|_{\mathcal{H}^{m-1}, -g'}+\|\nabla p\|_{H^{m-1}, -g} \\
\lesssim&\,  \|\nabla v\|_{\mathcal{H}^{m-1}}+\|\nabla p\|_{H^{m-1}, -g}.
\end{aligned}
\end{equation}
Therefore also $\mathbb{T}\eta = \eta-v-p \in \mathcal{M}^m$. Moreover, from \eqref{aux_System}, we have $\eta(0)=0$. Finally,
\begin{equation}
\begin{aligned}
& \frac{1}{\tau}c^2_g\Delta \psi+\frac{1}{\tau}b\Delta v + \frac{1}{\tau}\displaystyle%
\int_{0}^{\infty}g(s)\Delta\eta(s)\ds\\
=&\,  w+ \frac{1}{\tau}w+\frac{1}{\tau}(b-\tau c^2_g)v-h \in H^{m-1}(\R^n).
\end{aligned}
\end{equation}
This completes the proof.
\end{proof} 
\section{Short-time existence for the JMGT equation} \label{Sec:LocalExistence}
We claim that a unique solution to our nonlinear problem exists for sufficiently short final time. Since there is a quadratic gradient nonlinearity in the JMGT equation, we will need a bound on $\Vert \nabla\psi\Vert_{L^\infty}$ and $\Vert \nabla v \Vert_{L^\infty}$ in the upcoming estimates. We intend to employ the Sobolev embedding
\[\nabla \psi(t), \nabla v(t) \in H^r(\R^n) \hookrightarrow L^{\infty}(\R^n), \quad r>n/2. \] 
This leads to the assumption $m >n/2+1$ on the order of Hilbert space $\mathcal{H}^{m-1}$. 
\begin{theorem} \label{Thm:LocalExistence}
	Let $b\geq \tau c^2> \tau c^2_g$ and $n \geq 3$. Assume that $\Psi_0 \in \mathcal{H}^{m-1}$ for an integer $m>n/2+1$. 
	Then there exists a final time \[T=T(\|\Psi_0\|_{\mathcal{H}^{m-1}})\]
	such that problem   \eqref{Main_System},  \eqref{Main_System_IC} admits a unique mild solution 
	\begin{equation}
	\Psi=(\psi, v ,w, \eta)^T \in C([0,T]; \mathcal{H}^{m-1}),
	\end{equation} 
given by
\begin{equation} \label{IP_Psi0}
\Psi= e^{t\mathcal{A}}\Psi_0+\int_0^te^{(t-r)\mathcal{A}}\mathbb{F}(\Psi, \nabla \Psi)(r)\, \textup{d}r,
\end{equation}
where the functional $\mathbb{F}$ is defined in \eqref{def_F}. 
\end{theorem}
\begin{proof}
We prove the statement by employing the Banach fixed-point theorem; cf.~\cite{Racke_Said_2019, KaltenbacherNikolic,lasiecka2017global, kaltenbacher2012well} and \cite[Theorem 2.5.4]{zheng2004nonlinear}.  To this end, we define the ball
	\begin{equation} \label{B_L}
\begin{aligned}
\mathcal{B}_{L}=\{\Phi \in  C([0,T]; \mathcal{H}^{m-1}):    \,  \|\Phi(t)\|_{\mathcal{H}^{m-1}} \leq L, \ \forall t \in [0,T],   \ \Phi(0)=\Psi_0 \, \},
\end{aligned}
\end{equation}
equipped with the norm
\begin{equation}
\|\Phi\|_{\mathcal{B}_L}=\sup_{0\leq t \leq T} \|\Phi(t)\|_{\mathcal{H}^{m-1}}.
\end{equation}
The radius $L \geq \|\Psi_0\|_{ \mathcal{H}^{m-1}}$ of the ball will be conveniently chosen as large enough below. We note that $\mathcal{B}_L$ is a closed convex subset of $C([0,T]; \mathcal{H}^{m-1})$. \\
\indent For a given $\Phi=(\psi^\phi, v^\phi, w^\phi, \eta^\phi)^T$ in $\mathcal{B}_L$, we define the mapping $\mathcal{T}: \Phi \mapsto \Psi$ by
\begin{equation} \label{IP_Psi0}
\Psi= e^{t\mathcal{A}}\Psi_0+\int_0^te^{(t-r)\mathcal{A}}\mathbb{F}(\Phi, \nabla \Phi)(r)\, \textup{d}r.
\end{equation}
\indent We claim that $\mathcal{T}$ is a self-mapping and that it is strictly contractive.  We begin by proving that $\mathcal{T}(\mathcal{B}_L) \subset \mathcal{B}_L$. Let $\Phi \in \mathcal{B}_L$. Since $m-1>n/2$, we have
\begin{equation}
\begin{aligned}
\|f(t)\|_{H^{m-1}}=&\,\|2k v^\phi(t)w^\phi(t)+2\nabla \psi^\phi \cdot \nabla v^\phi\|_{H^{m-1}}  \\ 
\lesssim& \,\begin{multlined}[t] \|v^\phi(t)\|_{H^{m-1}}\|w^\phi(t)\|_{H^{m-1}}
+\|\nabla \psi^\phi(t)\|_{H^{m-1}}\|\nabla v^\phi(t)\|_{H^{m-1}}.
\end{multlined}
\end{aligned}
\end{equation}
An application of Young's inequality immediately yields
\begin{equation} \label{ineq_F}
\|f(t)\|_{H^{m-1}} \lesssim \|\Phi(t)\|^2_{\mathcal{H}^{m-1}} \lesssim L^2, \quad t \in [0,T].
\end{equation}
We are now ready to estimate $\Psi(t)$:
\begin{equation}
\begin{aligned}
\|\Psi(t)\|_{\mathcal{H}^{m-1}}\leq&\, \|e^{t\mathcal{A}}\Psi_0\|_{\mathcal{H}^{m-1}}+\int_0^t \|e^{(t-r)\mathcal{A}}\mathbb{F}(\Phi, \nabla \Phi)\|_{\mathcal{H}^{m-1}} \, \textup{d}r \\
\leq&\, \|\Psi_0\|_{\mathcal{H}^{m-1}}+\int_0^t \|\mathbb{F}(\Phi, \nabla \Phi)(r)\|_{\mathcal{H}^{m-1}} \, \textup{d}r \\
\leq&\, \|\Psi_0\|_{\mathcal{H}^{m-1}}+\int_0^t \|f(r)\|_{H^{m-1}} \, \textup{d}r \\
\lesssim&\,  \|\Psi_0\|_{\mathcal{H}^{m-1}}+T L^2,
\end{aligned}
\end{equation}
where we have employed inequality \eqref{ineq_F} in the last line. Therefore, there exists a positive constant $C_\star$ such that
\begin{equation} \label{smalness_nl_eq_Rn}
\begin{aligned}
\|\Psi\|_{\mathcal{B}_L} \leq& \, C_\star ( \|\Psi_0\|_{\mathcal{H}^{m-1}}+TL^2).
\end{aligned}
\end{equation}
The final time $T$ can then be chosen small enough and the radius $L$ of the ball $\mathcal{B}_L$ large enough so that $\Psi \in \mathcal{B}_L$. Indeed, we first fix $L$ so that 
\begin{eqnarray*}
C_\star \|\Psi_0\|_{\mathcal{H}^{m-1}} \leq \frac{L}{2}. 
\end{eqnarray*}
With $L$ fixed, we choose the time horizon $T>0$ small enough so that 
\begin{equation} \label{smallness_T}
T\leq \frac{1}{2C_\star L}.
\end{equation}
This choice of the radius $L$ and final time $T$ yields
\[\|\Psi\|_{\mathcal{B}_L} \leq L,\]
and, therefore,  $\mathcal{T}(\Phi) \in \mathcal{B}_L$.\\
\indent To prove contractivity, take any $\Phi=(\psi^{\phi}, v^{\phi}, w^{\phi}, \eta^{\phi})^T$ and $\Phi^\star=(\psi^{\phi}_\star, v^{\phi}_\star, w^{\phi}_\star, \eta^{\phi}_\star)^T$ in $\mathcal{B}_L$. We know that
\begin{equation}
\begin{aligned}
\|\mathcal{T}(\Phi)-\mathcal{T}(\Phi^\star)\|_{\mathcal{H}^{m-1}}\leq&\,  \int_0^t \|e^{(t-r)\mathcal{A}}\left[\mathbb{F}(\Phi, \nabla \Phi)-\mathbb{F}(\Phi^\star, \nabla \Phi^\star)\right]\|_{\mathcal{H}^{m-1}} \, \textup{d}r \\
\lesssim &\,\int_0^t \|v^{\phi} w^\phi-v^\phi_{\star} w_{\star}^\phi+\nabla \psi^{\phi} \cdot \nabla v^\phi-\nabla \psi^\phi_{\star} \cdot \nabla v_{\star}^\phi\|_{H^{m-1}} \, \textup{d}r.
\end{aligned}
\end{equation}
We can further estimate the last term as follows:
\begin{equation}
\begin{aligned}
&\|v^{\phi} w^\phi-v^\phi_{\star} w_{\star}^\phi+\nabla \psi^{\phi} \cdot \nabla v^\phi-\nabla \psi^\phi_{\star} \cdot \nabla v_{\star}^\phi\|_{H^{m-1}}\\
\lesssim&\, \begin{multlined}[t]\|v^\phi-v^\phi_\star\|_{H^{m-1}} \|w^\phi\|_{H^{m-1}}+\|v^\phi_\star\|_{H^{m-1}} \|w^\phi-w^\phi_\star\|_{H^{m-1}}\\[1mm]+\|\nabla (\psi^\phi -\psi^\phi_\star)\|_{H^{m-1}}\| \nabla v^\phi\|_{H^{m-1}} +\|\nabla \psi^\phi_{\star}\|_{H^{m-1}}\| \nabla (v^\phi-v^\phi_{\star})\|_{H^{m-1}},
\end{multlined}
\end{aligned}
\end{equation}
recalling that $m-1 >n/2$. Therefore, we have
\begin{equation}
\begin{aligned}
\|\mathcal{T}(\Phi)-\mathcal{T}(\Phi^\star)\|_{\mathcal{B}_L} \lesssim&  \,  \int_0^t \left(\|\Phi\|_{\mathcal{H}^{m-1}}+\|\Phi^\star\|_{\mathcal{H}^{m-1}} \right) \|\Phi -\Phi^\star\|_{\mathcal{H}^{m-1}} \,\textup{d}r\\[1mm]
\lesssim&\, T L \|\Phi -\Phi^\star\|_{\mathcal{B}_L}.
\end{aligned}
\end{equation}
By reducing the final time $T$, we can then guarantee that the mapping $\mathcal{T}$ is strictly contractive. On account of Banach's fixed-point theorem, a unique solution $\Psi=\Phi \in \mathcal{B}_L$ of the problem exists. Uniqueness in $C([0,T]; \mathcal{H}^{m-1})$ follows by assuming that $\Psi,  \Psi^\star \in C([0,T]; \mathcal{H}^{m-1})$ solve the problem and then applying Gronwall's inequality to
\begin{equation}
\begin{aligned}
\|\Psi-\Psi^\star\|_{\mathcal{H}^{m-1}}\lesssim&\,  \int_0^t \left(\|\Psi\|_{\mathcal{H}^{m-1}}+\|\Psi^\star\|_{\mathcal{H}^{m-1}} \right) \|\Psi -\Psi^\star\|_{\mathcal{H}^{m-1}} \,\textup{d}r.
\end{aligned}
\end{equation}
This finishes the proof of Theorem \ref{Thm:LocalExistence}.
\end{proof}%
\section{Energy estimates in the subcritical case} \label{Sec:EnergyEstimates}
\indent To prove global solvability, we intend to derive an energy estimate for the solution of the nonlinear problem that is uniform in time. Compared to~\cite{nikolic2020mathematical}, where local nonlinear effects in propagation were neglected, equation \eqref{Main_Equation} has a quadratic gradient nonlinearity. This means we have to involve higher-order energies in the estimates. Furthermore, we wish to derive a bound valid for all $n \geq 3$.\\ 
\indent To justify the upcoming estimates, we briefly consider again the regularity of $w$. According to Theorem~\ref{Thm:LocalExistence}, the nonlinear problem has the solution
\begin{equation}
(\psi, v ,w, \eta)^T \in C([0,T]; \mathcal{H}^{m-1})
\end{equation}
for sufficiently short final time. Looking at the third equation in the system, we know that
\begin{equation}
\begin{aligned}
\tau \|w_t\|_{H^{m-2}} \leq&\, \begin{multlined}[t] \|w\|_{H^{m-2}}+c^2_g \|\Delta \psi\|_{H^{m-2}}+b\|\Delta v\|_{H^{m-2}} \\+ 
\int_{0}^{\infty}g(s)\|\Delta\eta(s)\|_{H^{m-2}}\ds+2\|k vw+\nabla \psi \cdot \nabla v\|_{H^{m-2}}.
 \end{multlined}
\end{aligned}
\end{equation}
Similarly to \eqref{est_g}, we can further estimate the $\eta$ terms as follows:
\begin{equation}
\begin{aligned}
\int_{0}^{\infty}g(s)\|\Delta\eta(s)\|_{H^{m-2}}\ds \leq&\, \sqrt{c^2-c^2_g}\left(\int_{0}^{\infty}g(s)\|\Delta\eta(s)\|^2_{H^{m-2}}\ds \right)^{1/2}\\
\leq&\,  \sqrt{\frac{c^2-c^2_g}{\zeta}}\left(\int_{0}^{\infty}-g'(s)\|\Delta\eta(s)\|^2_{H^{m-2}}\ds \right)^{1/2}.
\end{aligned}
\end{equation}
Therefore, we have $w \in C^1([0,T]; H^{m-2}(\R^n)) \cap C([0,T]; H^{m-1}(\R^n))$ .\\
\indent  To derive higher-order bounds, we need to work with the space-differentiated system. For simplicity, we introduce the shorthand tilde notation
\begin{equation}
\tilde{\zeta}=\nabla^\kappa \zeta \quad \text{for}\quad \zeta \in{\{\psi, v,w,\eta \}}, \quad \kappa \geq 1.   
\end{equation}
 We then apply the operator $\nabla^\kappa $ ($\kappa\geq 1$) to the system \eqref{Main_System} to obtain
\begin{equation}  \label{Main_System_kappa}
\begin{cases}
\tilde{\psi}_{t}=\tilde{v}, \\ 
\tilde{v}_{t}=\tilde{w}, \\ 
\tau \tilde{w}_{t}=- \tilde{w}+c^2_g\Delta \tilde{\psi}+b\Delta \tilde{v} + \displaystyle%
\int_{0}^{\infty}g(s)\Delta\tilde{\eta}(s)\ds+
F^{(\kappa)}(\psi,v,\nabla \psi,\nabla v) 
,
\\ 
\tilde{\eta}_{t}=\tilde{v}-\tilde{\eta}_{s},%
\end{cases}%
\end{equation}
with the right-hand side nonlinearity given by 
\begin{equation}\label{F_k_Form}
F^{(\kappa)}(\psi,v,\nabla \psi,\nabla v)=2k[\nabla^\kappa,v]w+2k v\tilde{w}+2\kappa[\nabla^\kappa,\nabla \psi] \cdot \nabla v+2k\nabla \psi \cdot \nabla\tilde{v},
\end{equation}
recalling how the commutator is defined in \eqref{commutator}.\\
\indent We also introduce two new energy norms at this point. Let $n \geq 3$. For a given integer $p\geq 1$, we define 
\begin{equation} \label{EnergyNorm}
\begin{aligned}
\|\Psi\|^2_{\mathcal{E}_p(t)}
= \sup_{0\leq \sigma \leq t} \, \sum_{\kappa=0}^p \,  \mathcal{E}^{(\kappa)}[\Psi](\sigma), 
\end{aligned}
\end{equation}
where
\begin{equation} \label{E_kappa}
\begin{aligned}
\mathcal{E}^{(\kappa)}[\Psi] =  \begin{multlined}[t]\Vert\nabla^{\kappa+1}(\psi +\tau v)\Vert^{2}_{L^2} + \Vert \nabla^\kappa( v+\tau w)\Vert^{2}_{L^2}+\Vert \nabla^{\kappa+1} v\Vert^{2}_{L^2}
\\+\Vert \nabla^{\kappa+1}\eta\Vert^{2}_{L^2, -g'}
+\Vert \nabla^\kappa w(\sigma )\Vert _{L^{2}}^{2}\\
+\Vert\Delta\nabla^{\kappa}(\psi +\tau v)\Vert^{2}_{L^2} + \Vert\nabla^{\kappa+1}(v+\tau w)\Vert^{2}_{L^2}\\+\Vert \Delta \nabla^{\kappa} v\Vert^{2}_{L^2}
+\Vert \Delta\nabla^{\kappa}\eta\Vert^{2}_{L^2, -g'}.
\end{multlined}
\end{aligned}
\end{equation}
With the choice $p=m-2$ this norm is clearly equivalent to $\displaystyle \sup_{0\leq \sigma \leq t} \|\Psi(\sigma)\|_{\mathcal{H}^{m-1}}$. The corresponding dissipation norm is given by
\begin{equation} \label{DissipativeEnergyNorm}
\begin{aligned}
\|\Psi\|^2_{\mathcal{D}_p(t)}
=&\,\int_{0}^{t} \sum_{\kappa=0}^p \mathcal{D}^{(\kappa)}[\Psi] \, \textup{d} \sigma  \vphantom{\int_{0}^{t}},
\end{aligned}
\end{equation}
where we have set
\begin{equation} \label{D_kappa}
\begin{aligned}
\mathcal{D}^{(\kappa)}[\Psi] =&\, \begin{multlined}[t]\Vert \nabla^{\kappa+1}
v(\sigma )\Vert _{L^2}^{2}+\Vert \nabla^{\kappa+1} \eta (\sigma )\Vert _{L^2, -g'}^{2}  +\Vert\Delta\nabla^{\kappa}(\psi +\tau v)\Vert^{2}_{L^2}\\ + \Vert\nabla^{\kappa+1}(v+\tau w)\Vert^{2}_{L^2}+\Vert \Delta \nabla^{\kappa} v\Vert^{2}_{L^2}
+\Vert \Delta\nabla^{\kappa}\eta\Vert^{2}_{L^2, -g'} \\
+\Vert \nabla^\kappa w(\sigma )\Vert
_{L^{2}}^{2} . \end{multlined}
\end{aligned}
\end{equation}
\vspace*{1mm}
\indent Our overall goal in the remainder of this section is to prove that the energy $\|\Psi\|_{\mathcal{E}_{m-2}(t)}$ together with the related quantity  $\|\Psi\|_{\mathcal{D}_{m-2}(t)}$ are uniformly bounded for all time if the energy at initial time is sufficiently small. Following~\cite{Racke_Said_2019,nikolic2020mathematical}, we can achieve this by first proving that
\begin{equation}  \label{Estimate_Main_preliminary}
\|\Psi\|^2_{\mathcal{E}_{m-2}(t)}+\|\Psi\|^2_{\mathcal{D}_{m-2}(t)}\lesssim \|\Psi_0\|^2_{\mathcal{E}_{m-2}(0)}+\|\Psi\|_{\mathcal{E}_{m-2}(t)}\|\Psi\|^2_{\mathcal{D}_{m-2}(t)}.
\end{equation}
If $\|\Psi_0\|_{\mathcal{E}_{m-2}(0)}$ is small enough, we can then employ Lemma~\ref{Lemma_Stauss} to arrive at our claim. We prove this inequality in two parts: the first is dedicated to estimating the linear and the second to the right-hand side nonlinear terms in the equation. To achieve uniform stability, we assume that we are in the non-critical regime here, where $b>\tau c^2$.\\
\indent To simplify the notation involving the nonlinear terms, we introduce the right-hand side functionals $\mathcal{F}^{(\kappa)}_0$ and $\mathcal{F}^{(\kappa)}_1$ as
\begin{equation} \label{Def_Rhs}
\begin{aligned}
\mathcal{F}^{(\kappa)}_0(\varphi)= (F^{(\kappa)}, \varphi)_{L^2}, \quad \mathcal{F}^{(\kappa)}_1(\varphi)= (\nabla F^{(\kappa)}, \nabla \varphi)_{L^2},
\end{aligned}
\end{equation}
where  $\kappa \in \{0, 1, \ldots, m-2\}$, and $\varphi$ stands for different test functions that will be used below. \\
\indent We set aside estimates of the nonlinear terms for a moment and focus on estimating the linear terms in the equation.
\begin{proposition} \label{Prop:EnergyEst_LinearParts}
Let $b> \tau c^2>\tau c^2_g$. Then for each $\kappa \in \{0, 1, \ldots, m-2\}$, it holds 
	\begin{equation} \label{energy_est_kappa}
	\begin{aligned}
	&  \sup_{0\leq \sigma \leq t}  \mathcal{E}^{(\kappa)}[\Psi](\sigma)+  \int_0^t \mathcal{D}^{(\kappa)}[\Psi](\sigma) \, \textup{d}\sigma\\
	\lesssim&\, \begin{multlined}[t]  \mathcal{E}^{(\kappa)}[\Psi_0]
	+ \int_0^t \left\{ |\mathcal{F}_0^{(\kappa)}(\nabla^{\kappa}(v+\tau w))|+|\mathcal{F}_1^{(\kappa)}(\nabla^\kappa(v+\tau w))|\right.\\ \left.+|\mathcal{F}_0^{(\kappa)}(\nabla^\kappa w)|
	+|\mathcal{F}_1^{(\kappa)}(\nabla^{\kappa}(\psi+\tau v))| +|\mathcal{F}_1^{(\kappa)}(\tau \nabla^\kappa v)| \right\} \, \textup{d} \sigma. \end{multlined}
	\end{aligned}
	\end{equation}
\end{proposition}

\begin{proof}
The proof follows by performing the energy analysis of~\cite{nikolic2020mathematical} on the space-differentiated system \eqref{Main_System_kappa}, so we provide only an outline of the arguments here and refer to~\cite{nikolic2020mathematical} for details. Similarly to our previous reasoning, we first need to introduce energies that are tailored to our particular problem. The first one is given by
	\begin{equation} \label{energy}
	\begin{aligned}
	E_1^{(\kappa)}(t)
	=& \,\begin{multlined}[t] \dfrac{1}{2}\left [\vphantom{\int_{0}^{%
			\infty}} c^2_g%
	\|\nabla^{\kappa +1}(\psi+\tau v)\|_{L^2}^{2}+\tau(b -\tau
	c^2_g )\|\nabla^{\kappa+1} v\|^{2}_{L^2}\right.   \\
	\left. +\|\nabla^{\kappa}(v+\tau w)\|^{2}_{L^2}+\tau \Vert \nabla^{\kappa+1}
	\eta \Vert _{L^2, -g'}^{2}+ \|\nabla^{\kappa+1} \eta\|^2_{L^2, g}
	\right.\\
	\left.+ 2\tau \int_{\mathbb{R}^{n}}
	\int_{0}^{\infty}g(s) \nabla^{\kappa+1} \eta (s) \cdot \nabla^{\kappa+1} v
	\ds \dx\right] \end{multlined}
	\end{aligned}
	\end{equation}
	at time $t \geq 0$. Applying the energy estimates in~\cite[Proposition 4.1]{nikolic2020mathematical} to the system satisfied by $(\nabla^\kappa \psi, \nabla^\kappa v, \nabla^\kappa w, \nabla^\kappa \eta)$ yields the following dissipativity relation:
	\begin{equation} \label{Energy1_dissipation}
	\frac{\textup{d}}{\dt}E_1^{(\kappa)}(t)+( b -\tau c^2 )\Vert \nabla^{\kappa+1} v(t)\Vert_{L^2}^{2}+\frac{%
		1}{2}\Vert \nabla^{\kappa+1} \eta\Vert _{L^2, -g'}^{2}\leq \,
	|\mathcal{F}_0^{(\kappa)}(\nabla^\kappa (v + \tau  w))| 
	\end{equation}
	for all $t\geq 0$. Equipped with the arguments presented in the proof of Proposition~\ref{Prop:Equiv}, it is straightforward to check that this problem-specific energy is equivalent to 
	\begin{equation}\label{E_1_Eqv}
	\begin{aligned}
	\mathscr{E}^{(\kappa)}_1[\Psi]
	=&\, \begin{multlined}[t]\Vert\nabla^{\kappa+1}(\psi +\tau v)\Vert^{2}_{L^2} + \Vert \nabla^\kappa( v+\tau w)\Vert^{2}_{L^2}+\Vert \nabla^{\kappa+1} v\Vert^{2}_{L^2}
	+\Vert \nabla^{\kappa+1}\eta\Vert^{2}_{L^2, -g'}. 
	\end{multlined}
	\end{aligned}
	\end{equation} 
Secondly, we introduce the energy of order $\kappa+2$ by
	\begin{equation} \label{E_2}
	\begin{aligned}
	E_{2}^{(\kappa)}(t)
	=&\, \begin{multlined}[t]\frac{1}{2}\left[\vphantom{\int_{0}^{%
			\infty}} c^2_g\left\Vert \Delta
	( \nabla^{\kappa} (\psi+\tau v))\right\Vert _{L^{2}}^{2}+\tau(b -\tau c^2_g )\left\Vert
	\Delta \nabla ^\kappa v\right\Vert _{L^{2}}^{2}\right. \\ \left.+\Vert \nabla^{\kappa+1} (v+\tau w)\Vert
	_{L^2}^2+\tau \Vert \Delta \nabla^\kappa \eta \Vert^2_{L^2, -g'}    
	+\Vert \Delta \nabla^\kappa  \eta\Vert^2_{L^{2}, g} \right. \\ \left. +2\tau \int_{\mathbb{R}^{n}}\int_{0}^{%
		\infty}g(s)\Delta \nabla^\kappa v\,\Delta \nabla^\kappa \eta(s)\,\ds\dx\right]. \end{multlined}
	\end{aligned}
	\end{equation}
	This energy is clearly equivalent to $E_{1}^{(\kappa+1)}$ and thus to 
	\begin{equation}\label{E_2_Eqv}
	\begin{aligned}
	\mathscr{E}_2^{(\kappa)}[\Psi] =&\, \begin{multlined}[t] \Vert\Delta\nabla^{\kappa}(\psi +\tau v)\Vert^{2}_{L^2} + \Vert\nabla^{\kappa+1}(v+\tau w)\Vert^{2}_{L^2}\\+\Vert \Delta \nabla^{\kappa} v\Vert^{2}_{L^2}
	+\Vert \Delta\nabla^{\kappa}\eta\Vert^{2}_{L^2, -g'}.
	\end{multlined}
	\end{aligned}
	\end{equation}
It also satisfies a dissipative relation in the folowing form:
	\begin{equation}  \label{Energy2_dissipation}
	\frac{\textup{d}}{\dt}E_2^{(\kappa)}(t)+\left(b -\tau c^2 \right) \left\Vert
	\Delta \nabla ^{\kappa}v\right\Vert _{L^{2}}^{2}+\frac{1}{2}\Vert \Delta \nabla^\kappa\eta
	\Vert _{L^2, -g'}^{2}\leq \, |\mathcal{F}_1^{(\kappa)}(\nabla ^\kappa (v + \tau  w))|;
	\end{equation}%
	see~\cite[Proposition 4.2]{nikolic2020mathematical} for the proof in the case $\kappa=0$. To capture dissipation of terms $\|\nabla^{\kappa+1} (\psi+\tau v)\|_{L^2}$ and $\|\nabla^{\kappa+1}(v+\tau w)\|_{L^2}$, we introduce the functionals 
	\begin{equation} \label{F_Functionals}
	\begin{aligned} 
	F_{1}^{(\kappa)}(t)&=\,\int_{\mathbb{R}^{n}}\nabla^{\kappa+1} ( \psi+\tau v)\cdot \nabla^{\kappa+1} (v+\tau w)\dx, \\
	\end{aligned}
	\end{equation} 
	and
	\begin{equation} \label{F_Functionals}
	\begin{aligned} 
	F_{2}^{(\kappa)}(t)&=\,-\tau \int_{\mathbb{R}^{n}} \nabla^{\kappa+1} v  \cdot  \nabla^{\kappa+1} (v+\tau w) \dx, \
	\end{aligned}
	\end{equation} 
	for all $t \geq 0$, where  $\kappa \in \{0, 1, \ldots, m-2\}$. They satisfy suitable dissipative relations. It can be shown that 
	\begin{equation} \label{F_1_Estimate}
	\begin{aligned}
	&\frac{\textup{d}}{\dt}F_{1}^{(\kappa)}(t)+(c^2_g -\epsilon _{0}-(c^2 -c^2_g )\epsilon
	_{1})\Vert \Delta \nabla^{\kappa}(\psi+\tau v)\Vert _{L^{2}}^{2} \\
	\leq&\, \begin{multlined}[t] \Vert \nabla^{\kappa+1} (v+\tau w)\Vert _{L^{2}}^{2}+C(\epsilon
	_{0})\Vert \Delta \nabla^\kappa v\Vert _{L^{2}}^{2}
	+C(\epsilon _{1})\Vert \Delta\nabla^\kappa \eta\Vert _{L^{2},g}^{2}+|\mathcal{F}_1^{(\kappa)}(\nabla^{\kappa}(\psi + \tau v))|.\end{multlined}
	\end{aligned}
	\end{equation}
	for any positive $\epsilon _{0}$, $\epsilon _{1}>0$. Furthermore, for any $\epsilon_{2},\epsilon_{3}>0$,  we have 
	\begin{equation}\label{F_2_Estimate}
	\begin{aligned}
	&\frac{\textup{d}}{\dt}F_{2}^{(\kappa)}(t)+(1-\epsilon_{3})\Vert\nabla^{\kappa+1}(v+\tau w)\Vert^{2}_{L^2} \\
	\leq&\, \begin{multlined}[t]\epsilon_{2}\Vert\Delta\nabla^\kappa(\psi +\tau v)\Vert^{2}_{L^2}  +C(\epsilon_{3},\epsilon_{2})(\Vert \Delta\nabla^\kappa v\Vert^{2}_{L^2}+\Vert \nabla^{\kappa+1} v\Vert^{2}_{L^2}) \\
	+\frac{1}{2}\Vert \nabla^{\kappa+1}\eta\Vert^{2}_{L^2, g}+|\mathcal{F}_1^{(\kappa)}(\tau \nabla^\kappa v)|;\end{multlined}
	\end{aligned}
	\end{equation}
	see Lemmas 4.3 and 4.4 in \cite{nikolic2020mathematical} for the case $\kappa=0$, the arguments can easily be adapted to include $\kappa \in \{0, \ldots, m-2\}$. Testing the space-differentiated third equation in the system by $\nabla^\kappa w$ yields
	\begin{equation} \label{E_0_Energy}
	\begin{aligned}  
	&\frac{1}{2}\frac{\textup{d}}{\dt}\|\nabla^\kappa w(t)\|^2_{L^2}+\frac{1}{2}\|\nabla^\kappa w\|^2_{L^2}\\
	\lesssim&\,
	\Vert
	\Delta \nabla^\kappa (\psi+\tau v)\Vert_{L^2}^2+\Vert \Delta \nabla^\kappa v\Vert _{L^{2}}^2+\Vert\Delta\nabla^\kappa \eta\Vert^{2}_{L^2, g}+|\mathcal{F}_0^{(\kappa)}(\nabla^\kappa w)|.
	\end{aligned}
	\end{equation}
	\noindent We are ready to define the Lyapunov functional 
	$\mathcal{L}^{(\kappa)}$ of order $\kappa$ as 
	\begin{equation}\label{Lyapunov_k}
	\mathcal{L}^{(\kappa)}(t)=L_1 (E_1^{(\kappa)}(t)+E_2^{(\kappa)}(t)+\varepsilon\tau \Vert \nabla ^{\kappa}w(t)\Vert^2 _{L^{2}})+F_1^{(\kappa)}(t)+L_2F_2^{(\kappa)}(t),     
	\end{equation}
	at time $t \geq 0$. There exist a constant $L_1>0$ large enough and a constant $\varepsilon>0$ small enough such that the Lyapunov functional, defined in \eqref{Lyapunov_k},  satisfies for any $\kappa\geq 0$
	\begin{equation}\label{Lyap_Main_k}
	\begin{aligned}
	&\begin{multlined}[t]\frac{\textup{d}}{\dt}\mathcal{L}^{(\kappa)}(t)+ \Vert \nabla^{\kappa+1}
	v(t)\Vert _{L^2}^{2}+\Vert \nabla^{\kappa+1} \eta \Vert _{L^2, -g'}^{2}+\mathscr{E}^{(\kappa)}_2[\Psi](t)
	+\Vert\nabla^{\kappa} w(t)\Vert
	_{L^{2}}^{2} \end{multlined}\\[1mm]
	\lesssim&\,\begin{multlined}[t]   |\mathcal{F}_0^{(\kappa)}(\nabla^{\kappa}(v+\tau w))|+|\mathcal{F}_1^{(\kappa)}(\nabla^\kappa(v+\tau w))|
	+|\mathcal{F}_0^{(\kappa)}(\nabla^\kappa w)| \\
	+|\mathcal{F}_1^{(\kappa)}(\nabla^{\kappa}(\psi+\tau v))| +|\mathcal{F}_1^{(\kappa)}(\tau \nabla^\kappa v)|,\end{multlined}
	\end{aligned}
	\end{equation}
	for all $t\in [0,T]$.  We note that having $b>\tau c^2$ in \eqref{Energy1_dissipation} and \eqref{Energy2_dissipation} is essential in obtaining dissipativity of this Lyapunov functional. \\
	\indent By integrating estimate \eqref{Lyap_Main_k} over the time interval $(0, \sigma)$ for $\sigma \in (0,t)$ and then taking the supremum over time, we obtain \eqref{energy_est_kappa}.
\end{proof}
\indent The remaining challenge in the uniform energy analysis is to control the terms $\mathcal{F}_0^{(\kappa)}$ and $\mathcal{F}_1^{(\kappa)}$ in estimate \eqref{energy_est_kappa}. To formulate the next results, we introduce here the functional
\begin{equation} \label{Lambda}
\begin{aligned}
\Lambda[\Psi](t)=\sup_{0\leq \sigma\leq t } & \begin{multlined}[t]\left(\vphantom{\Vert \nabla^2 \psi(\sigma)\Vert_{{H}^{\frac{n-2}{2}}}} \Vert v(\sigma)\Vert _{W^{1,\infty }}+\Vert w(\sigma)\Vert _{L^{\infty }}+\Vert
\nabla \psi(\sigma)\Vert _{L^{\infty }}\right.
\\[1mm] \left.\quad+\Vert \nabla \psi(\sigma)\Vert_{{H}^{\frac{n-2}{2}}}+\Vert \nabla^2 \psi(\sigma)\Vert_{{H}^{\frac{n-2}{2}}}+\Vert v(\sigma)\Vert_{{H}^{\frac{n-2}{2}}}
\right.\\[1mm] \left.+\Vert \nabla v(\sigma)\Vert_{{H}^{\frac{n-2}{2}}}+\Vert w(\sigma)\Vert_{{H}^{\frac{n-2}{2}}}\right). \end{multlined}
\end{aligned}
\end{equation}
We prove the estimates for the cases $\kappa=0$ and $\kappa \in \{1, \ldots, m-2\}$ separately.
\begin{theorem}\label{Thm:First_Eng_Estimate}
Let $b> \tau c^2>\tau c^2_g$ and $n\geq 3$. Then it holds that  
\begin{equation}  \label{Main_Estimate_D_0}
\|\Psi\|^2_{\mathcal{E}_0(t)}+\|\Psi\|^2_{\mathcal{D}_0(t)}\lesssim 
\|\Psi_0\|^2_{\mathcal{E}_0(0)} +\Lambda[\Psi](t)\|\Psi\|^2_{\mathcal{D}_0(t)}.
\end{equation}
\end{theorem}
\begin{proof}
To prove the statement, we should estimate the integral terms on the right-hand side of \eqref{energy_est_kappa}. Using H\"{o}lder's inequality yields
\begin{equation}\label{I_1_I_2}
\begin{aligned}
\left\vert \mathcal{F}_0^{(0)}(v+\tau w)\right\vert =&\, \left\vert 2k \int_{\mathbb{R}^{n}}\left( vw+\nabla \psi \cdot \nabla v\right) \left( v+\tau w\right) \dx\right\vert \\
\lesssim & \, \int_{\mathbb{R}^n} \left | vw(v+\tau w) \right|\dx+\int_{\mathbb{R}^n} \left| \nabla \psi \cdot \nabla
v \left( v+\tau w\right) \right| \dx.
\end{aligned}
\end{equation}
To estimate the terms on the right, we begin with noting that
\begin{equation}
\int_{\mathbb{R}^n} \left |  wv^2 \right| \dx \lesssim \Vert w\Vert_{L^2}\Vert v\Vert_{L^n}\Vert v\Vert_{L^{\frac{2n}{n-2}}}
\lesssim \Vert w\Vert_{L^2}\Vert v\Vert_{{H}^{\frac{n-2}{2}}} \Vert  \nabla v\Vert_{{L}^2},
\end{equation}   
where we have used the estimate \eqref{trilinear_est} for tri-linear terms, the endpoint Sobolev embedding \eqref{GNS_ineq}, and the embedding \eqref{embedding_2}.  Consequently, we obtain 
  \begin{equation}\label{BMO_First_Term}
 \begin{aligned} 
\int_0^t\int_{\mathbb{R}^n} \left |  wv^2 \right| \dx \, \textup{d}\sigma  
\lesssim&\, \sup_{0\leq \sigma\leq t}\Vert v(\sigma)\Vert_{{H}^{\frac{n-2}{2}}} \int_0^t (\Vert w(\sigma)\Vert_{L^2}^2
+\Vert \nabla v(\sigma)\Vert_{L^2}^2)\, \textup{d}\sigma \\
\lesssim&\, \sup_{0\leq \sigma\leq t}\Vert v(\sigma)\Vert_{{H}^{\frac{n-2}{2}}}\|\Psi\|^2_{\mathcal{D}_0(t)}.
\end{aligned}
\end{equation}
Similarly,  we have 
 \begin{equation}
\left\vert\int_{\mathbb{R}^n} vw^2 \dx\right\vert \lesssim \Vert w\Vert_{L^2}\Vert v\Vert_{L^n}\Vert w\Vert_{L^{\frac{2n}{n-2}}}
\lesssim \Vert w\Vert_{L^2}\Vert v\Vert_{{H}^{\frac{n-2}{2}}} \Vert \nabla  w\Vert_{L^2},
\end{equation}
from which it follows that 
\begin{equation}\label{BMO_Second}
\begin{aligned}
\int_0^t \int_{\mathbb{R}^n}\left\vert vw^2 \dx\right\vert  \, \textup{d}\sigma\
\lesssim \, \sup_{0\leq \sigma\leq t}\Vert v(\sigma)\Vert_{{H}^{\frac{n-2}{2}}}\|\Psi\|^2_{\mathcal{D}_0(t)}. 
\end{aligned}
\end{equation}
By adding estimates \eqref{BMO_First_Term} and \eqref{BMO_Second}, we find that
\begin{equation}\label{I_1_Estimate_Main}
\int_0^t \int_{\mathbb{R}^n}  \left | vw(v+\tau w) \right| \dx  \textup{d} \sigma \lesssim \sup_{0\leq \sigma\leq t}\Vert v(\sigma)\Vert_{{H}^{\frac{n-2}{2}}}\|\Psi\|^2_{\mathcal{D}_0(t)}.
\end{equation}
For the remaining terms on the right in \eqref{I_1_I_2}, we have by H\"older's inequality
\begin{equation} \label{ineq}
\int_0^t \int_{\mathbb{R}^n} |\nabla \psi \cdot \nabla v (\tau w)| \dx \textup{d}\sigma \lesssim \sup_{0\leq \sigma\leq t}\Vert \nabla \psi(\sigma)\Vert_{L^\infty}\|\Psi\|^2_{\mathcal{D}_0(t)}
\end{equation}
We also note that
\begin{equation}
\begin{aligned}
\int_{\mathbb{R}^n}| v\nabla \psi \cdot \nabla v| \dx \lesssim \, \Vert \nabla v\Vert_{L^2}\Vert \nabla \psi\Vert_{L^n}\Vert v\Vert_{L^{\frac{2n}{n-2}}}
\lesssim \,  \Vert \nabla v\Vert_{L^2}\Vert \nabla \psi\Vert_{{H}^{\frac{n-2}{2}}} \Vert \nabla v\Vert_{L^2},
\end{aligned}
\end{equation}
where we have again used the estimate \eqref{trilinear_est}, the endpoint Sobolev embedding \eqref{GNS_ineq}, and the embedding \eqref{embedding_2}. This immediately yields 
\begin{equation}\label{Estimate_J_1}
\int_0^t \int_{%
	\mathbb{R}^n} |v\nabla \psi \cdot \nabla v | \dx \textup{d}s \lesssim \sup_{0\leq \sigma\leq t}\Vert\nabla \psi(\sigma)\Vert_{{H}^{\frac{n-2}{2}}}\|\Psi\|^2_{\mathcal{D}_0(t)}.  
\end{equation}
From  estimates \eqref{ineq} and \eqref{Estimate_J_1}, we conclude that 
\begin{equation}\label{I_2_Estimate_Main}
\begin{aligned}
\int_0^t \int_{\mathbb{R}^n} |\nabla \psi \cdot \nabla
v \left( v+\tau w\right)| \dx \textup{d}\sigma\lesssim \sup_{0\leq \sigma\leq t}\Big(\Vert \nabla \psi(\sigma)\Vert_{L^\infty}+\Vert\nabla \psi(\sigma)\Vert_{{H}^{\frac{n-2}{2}}}\Big)%
\|\Psi\|^2_{\mathcal{D}_0(t)}. 
\end{aligned}
\end{equation}
Altogether, we have proven that
\begin{equation}
\begin{aligned}  \label{R_1_Estimate}
&\int_0^t | \mathcal{F}_0^{(0)}(v+\tau w)(\sigma)|\, \textup{d}\sigma 
\\
\lesssim&\,
\sup_{0\leq \sigma\leq t}\Big(\Vert \nabla \psi(\sigma)\Vert_{L^\infty}+\Vert\nabla \psi(\sigma)\Vert_{{H}^{\frac{n-2}{2}}}+\Vert v(\sigma)\Vert_{{H}^{\frac{n-2}{2}}}\Big)%
\|\Psi\|^2_{\mathcal{D}_0(t)}.
\end{aligned}
\end{equation}
We can estimate the term $\int_0^t  \mathcal{F}_0^{(0)}( w)(\sigma) \, \textup{d}\sigma$ in \eqref{energy_est_kappa} similarly, so we move on to estimating the remaining three $\mathcal{F}_1^{(0)}$ terms. We can rewrite $\mathcal{F}_1^{(0)}(v+\tau w)$ as
\begin{equation}
\begin{aligned}
&\mathcal{F}_1^{(0)}(v+\tau w) \\
=&\, 2k\int_{\mathbb{R}^{n}}\nabla \left( \dfrac{1}{\tau }v\left( v+\tau
w-v\right) +2\nabla \left( \psi+\tau v-\tau v\right)\cdot \nabla v\right) \cdot \nabla
(v+\tau w)\dx \\
=&\, \begin{multlined}[t]2k\int_{\mathbb{R}^{n}}\left( \dfrac{1}{\tau }v\nabla \left( v+\tau
w\right) +\dfrac{1}{\tau }\nabla v\left( v+\tau w\right) -\nabla \left\vert
v\right\vert ^{2}\right)\cdot \nabla (v+\tau w)\dx \\
+\int_{\mathbb{R}^{n}}\left( 2H(\psi+\tau v)\nabla v+2H(v)\nabla \left(
\psi+\tau v\right) -4\tau H(v)\nabla v\right)\cdot \nabla (v+\tau w)\dx, \end{multlined}
\end{aligned}
\end{equation}%
where we have introduced the following notation for the Hessian matrix of $f$:
\[H(f)=(\partial _{x_{i}}\partial _{x_{j}}f)_{,1\leq i,j\leq n}.\] 
By using the identity \[\left\Vert H\left( f\right)
\right\Vert _{L^{2}}= \left\Vert \Delta f\right\Vert _{L^{2}},\]
 together
with H\"{o}lder's inequality, we infer
\begin{equation}
\begin{aligned}
\vert\mathcal{F}_1^{(0)} (v+\tau w)\vert 
\lesssim& \, \begin{multlined}[t]\Vert \nabla v\Vert
_{L^2}( \Vert \nabla ( v+\tau w) \Vert
_{{H}^{\frac{n-2}{2}}}+\Vert \nabla v\Vert _{{H}^{\frac{n-2}{2}}})\Vert \nabla (v+\tau w)\Vert _{L^{2}}\\[1mm]
 +\Vert \nabla(v+\tau
w)\Vert _{L^2}^2\Vert \nabla v\Vert _{{H}^{\frac{n-2}{2}}} \\[1mm]
+ \Vert \nabla v\Vert _{L^{\infty }}\left( \Vert
\Delta (\psi+\tau v)\Vert _{L^{2}}+\Vert \Delta v\Vert
_{L^{2}}\right)\Vert \nabla (v+\tau
w)\Vert _{L^{2}}\\[1mm] +\Vert \nabla (\psi+\tau v)\Vert _{L^{\infty
}}\Vert \Delta v\Vert _{L^{2}} \Vert \nabla (v+\tau
w)\Vert _{L^{2}}.\end{multlined}
\end{aligned}
\end{equation}
Integrating the above inequality from $0$ to $t$, yields 
\begin{equation}  
\begin{aligned}
\label{R_2_Estimate}
&\int_0^t |\mathcal{F}_1^{(0)}(v+\tau w)(\sigma)|\, \textup{d}\sigma\\
\lesssim&\,\begin{multlined}[t] \sup_{0\leq \sigma\leq t}(\left\Vert
\nabla v(\sigma)\right\Vert _{{H}^{\frac{n-2}{2}}}+\left\Vert\nabla
v(\sigma)\right\Vert _{L^{\infty }}
+\left\Vert\nabla (v+\tau w)(\sigma)\right\Vert
_{{H}^{\frac{n-2}{2}}}\\+\left\Vert \nabla (\psi+\tau v)(\sigma)\right\Vert _{L^{\infty }})%
\|\Psi\|^2_{\mathcal{D}_0(t)}. \end{multlined}
\end{aligned}
\end{equation}
Similarly, we have the following bound:
\begin{equation}  \label{R_2_tilde}
\begin{aligned}
\int_0^t |\mathcal{F}_1^{(0)}(\psi+\tau v)(\sigma)|\, \textup{d}\sigma
\lesssim\,\sup_{0\leq \sigma\leq t}\big(\left\Vert
w(\sigma)\right\Vert _{{H}^{\frac{n-2}{2}}}+\left\Vert \nabla \psi (\sigma)\right\Vert _{L^{\infty
}}\big)\|\Psi\|^2_{\mathcal{D}_0(t)},
\end{aligned}
\end{equation}
as well as 
\begin{equation}  \label{R_2_tilde_}
\begin{aligned}
\int_0^t |\mathcal{F}_1^{(0)}(\tau v)(\sigma)|\, \textup{d}\sigma\
\lesssim\,\sup_{0\leq \sigma\leq t}\big(\left\Vert
w(\sigma)\right\Vert _{{H}^{\frac{n-2}{2}}}+\left\Vert \nabla \psi (\sigma)\right\Vert _{L^{\infty
}}\big)\|\Psi\|^2_{\mathcal{D}_0(t)}.
\end{aligned}
\end{equation}
By plugging estimates \eqref{R_1_Estimate}--\eqref{R_2_tilde_} into \eqref{energy_est_kappa}, we obtain the desired bound. 
\end{proof}
\noindent We next prove an analogous result when  $\kappa \in \{1, \ldots, m-2\}$.
\begin{theorem}\label{Energy_high_Order}
Let $b> \tau c^2>\tau c^2_g$ and $n \geq 3$. For any $t\geq 0$ and for any $\kappa \in \{1, \ldots, m-2\}$, it holds that  
\begin{equation}\label{Energy_k_N_Main_Estimate}
\begin{aligned}
 &\sup_{0\leq \sigma \leq t}  \mathcal{E}^{(\kappa)}[\Psi](\sigma)+  \int_0^t \mathcal{D}^{(\kappa)}[\Psi](\sigma) \, \textup{d}\sigma \\
\lesssim&\,\mathcal{E}^{(\kappa)}[\Psi](0)+\Lambda[\Psi](t)\int_0^t (\mathcal{D}^{(\kappa-1)}[\Psi](\sigma) \, +  \mathcal{D}^{(\kappa)}[\Psi](\sigma)) \, \textup{d}\sigma,
\end{aligned}
\end{equation}
where the function $\Lambda[\Psi]$ is defined in \eqref{Lambda}.  The functionals $\mathcal{E}^{(\kappa)}[\Psi]$ and $\mathcal{D}^{(\kappa)}[\Psi]$ are defined in \eqref{E_kappa} and \eqref{D_kappa}, respectively. 
\end{theorem}
\begin{proof}
As in the previous proof, the crucial part is to  get appropriate estimates on the five integral terms on the right-hand side of \eqref{energy_est_kappa}. Here we have to rely on the commutator bounds.
We begin by noting that
\begin{equation}\label{I_1_Terms}
\begin{aligned}
&|\mathcal{F}_0^{(\kappa)}(\nabla^{\kappa}(v+\tau w))(\sigma)|=|\mathcal{F}_0^{(\kappa)}((\tilde{v}+\tau \tilde{w}))(\sigma)|\\
\lesssim& \, \int_{\mathbb{R}^n} |[\nabla ^{\kappa},v]w| |\left( \tilde{v}+\tau \tilde{w}\right)| \dx+\int_{%
\mathbb{R}^n} |v\tilde{w}| |\left( \tilde{v}+\tau \tilde{w}\right)|\dx \\
&+\int_{\mathbb{R}^n}|[\nabla ^{\kappa},\nabla \psi]\nabla v||\left( \tilde{v}+\tau
\tilde{w}\right)|\dx + \int_{\mathbb{R}^n} |\nabla \psi \cdot \nabla \tilde{v}||\left( \tilde{v}+\tau \tilde{w}\right)|\dx.
\end{aligned}
\end{equation}
Starting from the last term on the right, we have
\begin{equation}  \label{T_4}
\begin{aligned}
\int_{\mathbb{R}^n} |\nabla \psi \cdot \nabla \tilde{v}||\left( \tilde{v}+\tau \tilde{w}\right)|\dx \leq \int_{\mathbb{R}^n} |\tilde{v}\nabla \psi \cdot \nabla \tilde{v }| \dx+\int_{\mathbb{R}%
^n} |\tau\nabla \psi \cdot \nabla \tilde{v} \tilde{w}| \dx.
\end{aligned}
\end{equation}
We can bound the second term on the right as follows:
\begin{eqnarray*}
\int_{\mathbb{R}%
	^n} |\tau\nabla \psi \cdot \nabla \tilde{v} \tilde{w}| \dx \lesssim \Vert \nabla \psi\Vert_{L^\infty}\Vert \nabla
\tilde{v}\Vert_{L^2}\Vert \tilde{w}\Vert_{L^2},
\end{eqnarray*}
which by H\"older's inequality yields
\begin{equation}\label{J_2_k_Main}
\int_0^t \int_{\mathbb{R}%
	^n} |\tau\nabla \psi \cdot \nabla \tilde{v} \tilde{w} \vert \dx \textup{d}\sigma\lesssim \sup_{0\leq \sigma \leq t}\Vert \nabla
\psi(\sigma)\Vert_{L^\infty}\int_0^t  \mathcal{D}^{(\kappa)}[\Psi](\sigma) \, \textup{d}\sigma.
\end{equation}
We next have 
\begin{equation}
\begin{aligned}
\int_{\mathbb{R}^n} |\tilde{v}\nabla \psi \cdot \nabla \tilde{v }| \dx \lesssim \, \Vert \tilde{v}\Vert_{L^\frac{2n}{n-2}}\Vert \nabla \psi\Vert_{L^n}\Vert \nabla
\tilde{v}\Vert_{L^{2}}
\lesssim& \, \Vert 
\nabla \tilde{v}\Vert_{L^2}\Vert \nabla \psi\Vert_{L^n}\Vert \nabla
\tilde{v}\Vert_{L^{2}},
\end{aligned}
\end{equation}
which leads to 
\begin{equation}\label{J_1_k_1}
\begin{aligned}
&\int_0^t \int_{\mathbb{R}^n} |\tilde{v}\nabla \psi \cdot \nabla \tilde{v } |\dx \textup{d} \sigma\\
\lesssim&\, \sup_{0\leq \sigma\leq t} \Vert \nabla \psi(\sigma)\Vert_{{H}^{\frac{n-2}{n}}}\int_0^t (\mathcal{D}^{(\kappa-1)}[\Psi](\sigma) \, +  \mathcal{D}^{(\kappa)}[\Psi](\sigma)) \, \textup{d}\sigma.
\end{aligned}
\end{equation}
We deduce from the derived bounds that
\begin{equation}\label{T_4_Estimate}
\begin{aligned}
&\int_0^t \int_{\mathbb{R}^n} |\nabla \psi \cdot \nabla \tilde{v}||\left( \tilde{v}+\tau \tilde{w}\right)|\dx\, \textup{d}\sigma\\
\lesssim& \, \sup_{0\leq \sigma\leq t}\big(\Vert \nabla
\psi(\sigma)\Vert_{L^\infty}+\Vert \nabla \psi(\sigma)\Vert_{{H}^{\frac{n-2}{2}}}\big)\int_0^t (\mathcal{D}^{(\kappa-1)}[\Psi](\sigma) \, +  \mathcal{D}^{(\kappa)}[\Psi](\sigma)) \, \textup{d}\sigma.
\end{aligned}
\end{equation}
We next want to estimate the second term on the right in \eqref{I_1_Terms}:
\begin{equation}\label{T_2_Terms}
\begin{aligned}
\int_{%
	\mathbb{R}^n} |v\tilde{w}| |\left( \tilde{v}+\tau \tilde{w}\right)|\dx =& \, \int_{\mathbb{R}^{n}}|v\tilde{w}||\left( \tilde{v}+\tau \tilde{w}\right) |\dx \\ 
\leq& \, \int_{\mathbb{R}^{n}}|v\tilde{w}\tilde{v}|\dx+\tau \int_{\mathbb{R}^{n}}|v\tilde{w}^{2}|\dx.
\end{aligned}  
\end{equation}  
We further have
\begin{equation}
\begin{aligned}
\tau \int_{\mathbb{R}^{n}}|v\tilde{w}^{2}|\dx \lesssim   \Vert \tilde{w}\Vert_{L^\frac{2n}{n-2}}\Vert v\Vert_{L^n}\Vert 
\tilde{w}\Vert_{L^{2}}
 \lesssim& \, \Vert 
\nabla \tilde{w}\Vert_{L^2}\Vert v\Vert_{L^n}\Vert 
\tilde{w}\Vert_{L^{2}},
\end{aligned}
\end{equation}
which  yields 
\begin{equation}\label{J_k_4_Main}
\int_0^t\tau \int_{\mathbb{R}^{n}}|v\tilde{w}^{2}|\dx \textup{d}\sigma\lesssim \sup_{0\leq \sigma\leq t} \Vert v(\sigma)\Vert_{{H}^{\frac{n-2}{2}}}\int_0^t   \mathcal{D}^{(\kappa)}[\Psi](\sigma) \, \textup{d}\sigma.
\end{equation}
We can derive the same bound for the first term on the right in \eqref{T_2_Terms} after integration in time. Consequently, we have the following estimate:
\begin{equation}\label{T_2_Main_N_2}
\int_0^t \int_{\mathbb{R}^n} |v\tilde{w}| |\left( \tilde{v}+\tau \tilde{w}\right)|\dx \,  \textup{d}\sigma\lesssim  \sup_{0\leq \sigma \leq
t} \Vert v(\sigma)\Vert_{{H}^{\frac{n-2}{2}}} \int_0^t   \mathcal{D}^{(\kappa)}[\Psi](\sigma) \, \textup{d}\sigma.
\end{equation}
We next  estimate the first term on the right-hand side in \eqref{I_1_Terms}:
\begin{equation}
\begin{aligned}
 \int_{\mathbb{R}^n} |[\nabla ^{\kappa},v]w| |\left( \tilde{v}+\tau \tilde{w}\right)| \dx \lesssim& \,
\left\Vert \lbrack \nabla ^{\kappa},v]w\right\Vert _{L^{\frac{2n}{n+2}}}\left\Vert \left(
\tilde{v}+\tau \tilde{w}\right) \right\Vert _{L^{\frac{2n}{n-2}}}\\
\lesssim& \, \left\Vert \lbrack \nabla ^{\kappa},v]w\right\Vert _{L^{\frac{2n}{n+2}}}  ( \Vert 
\tilde{v}+\tau \tilde{w} \Vert _{L^{2}}+\Vert \nabla(
\tilde{v}+\tau \tilde{w}) \Vert _{L^{2}}).
\end{aligned}
\end{equation}
Thus, we have by applying the Sobolev embedding \eqref{embedding_2}
\begin{equation}\label{Com_1_N_2}
\begin{aligned}
\left\Vert \lbrack \nabla ^{\kappa},v]w\right\Vert _{L^{\frac{2n}{n+2}}}\lesssim& \, (\Vert \nabla
v\Vert _{L^{n}}\Vert \nabla ^{\kappa-1}w\Vert _{L^{2}}+\Vert w\Vert _{L^{n}}\Vert
\nabla ^{\kappa}v\Vert _{L^{2}})\notag\\
\lesssim& \, (\Vert \nabla v\Vert_{{H}^{\frac{n-2}{2}}} +\Vert w\Vert_{{H}^{\frac{n-2}{2}}} ) (\Vert \nabla ^{\kappa-1}w\Vert _{L^{2}}+\Vert
\nabla ^{\kappa}v\Vert _{L^{2}}).
\end{aligned}
\end{equation}
Therefore,  we obtain  
\begin{equation}\label{T_1_Main_Estimate_N_2}
\begin{aligned}
&\int_{0}^{t} \int_{\mathbb{R}^n} |[\nabla ^{\kappa},v]w| |\left( \tilde{v}+\tau \tilde{w}\right)| \dx\, \textup{d}\sigma\\
 \lesssim&\, \sup_{0\leq \sigma \leq
t} \left(\Vert \nabla v(\sigma)\Vert_{{H}^{\frac{n-2}{2}}}+\Vert w(\sigma)\Vert_{{H}^{\frac{n-2}{2}}}\right) \int_0^t (\mathcal{D}^{(\kappa-1)}[\Psi](\sigma) \, +  \mathcal{D}^{(\kappa)}[\Psi](\sigma)) \, \textup{d}\sigma.
\end{aligned}
\end{equation}
 It remains to estimate the third term on the right in \eqref{I_1_Terms}: 
 \begin{equation}\label{T_3_N_2_5}
 \begin{aligned}
&\int_{\mathbb{R}^n}|[\nabla ^{\kappa},\nabla \psi]\nabla v||\left( \tilde{v}+\tau
\tilde{w}\right)|\dx \\[1mm]
\lesssim& \, \Vert \lbrack \nabla ^{\kappa},\nabla \psi]\nabla v\Vert _{L^{\frac{2n}{n+2}}}\left\Vert \left(
\tilde{v}+\tau \tilde{w}\right) \right\Vert _{L^{\frac{2n}{n-2}}}\\[1mm]
\lesssim& \, \left\Vert \lbrack \nabla ^{\kappa},\nabla \psi]\nabla v\right\Vert _{L^{\frac{2n}{n+2}}} (\Vert 
\tilde{v}+\tau \tilde{w} \Vert _{L^{2}}+\Vert \nabla(
\tilde{v}+\tau \tilde{w}) \Vert _{L^{2}}).
\end{aligned}
\end{equation}
Now, by applying once more the embedding $H^{\frac{n-2}{2}}(\R^n) \hookrightarrow L^n(\R^n)$, we have 
\begin{equation}\label{Estimate_N_2_T_3}
\begin{aligned}
\left\Vert \lbrack \nabla ^{\kappa},\nabla \psi]\nabla v\right\Vert _{L^{\frac{2n}{n+2}}}\lesssim& \, (\Vert \nabla^2
\psi\Vert _{L^{n}}\Vert \nabla ^{\kappa}v\Vert _{L^{2}}+\Vert \nabla v\Vert _{L^{n}}\Vert
\nabla ^{\kappa+1}\psi\Vert _{L^{2}})\\
\lesssim& \, (\Vert \nabla^2
\psi\Vert _{{H}^{\frac{n-2}{2}}}+\Vert \nabla v\Vert  _{{H}^{\frac{n-2}{2}}})(\Vert \nabla ^{\kappa}v\Vert _{L^{2}}+\Vert
\nabla ^{\kappa+1}\psi\Vert _{L^{2}}). 
\end{aligned}
\end{equation}
Hence, we obtain from the above two estimates
\begin{equation}\label{T_3_Estimate_N_2}
\begin{aligned}
&\int_0^t \int_{\mathbb{R}^n}|[\nabla ^{\kappa},\nabla \psi]\nabla v||\left( \tilde{v}+\tau
\tilde{w}\right)|\dx\, \textup{d}\sigma \\
\lesssim&\, \sup_{0\leq \sigma \leq
t} \left(\Vert \nabla^2 \psi(\sigma)\Vert_{{H}^{\frac{n-2}{2}}}+\Vert \nabla v(\sigma)\Vert_{{H}^{\frac{n-2}{2}}}\right)  \int_0^t (\mathcal{D}^{(\kappa-1)}[\Psi](\sigma) \, +  \mathcal{D}^{(\kappa)}[\Psi](\sigma)) \, \textup{d}\sigma.
\end{aligned}
\end{equation}
Altogether, we have proven the following estimate for $\mathcal{F}_0^{(\kappa)}(\nabla^{\kappa}(v+\tau w))$:
\begin{equation}\label{I_1_Main_Est_N_2}
\int_0^t |\mathcal{F}_0^{(\kappa)}(\nabla^{\kappa}(v+\tau w))(\sigma)|\, \textup{d}\sigma\lesssim\Lambda[\Psi](t) \int_0^t (\mathcal{D}^{(\kappa-1)}[\Psi](\sigma) \, +  \mathcal{D}^{(\kappa)}[\Psi](\sigma)) \, \textup{d}\sigma,
\end{equation}
with the function $\Lambda[\Psi]$ defined in \eqref{Lambda}.\\
\indent Next we claim that
\begin{equation}  \label{I_2_4_D_k}
\int_0^t (|\mathcal{F}_1^{(\kappa)}(\nabla^\kappa(v+\tau w))(\sigma)|+|\mathcal{F}_1^{(\kappa)}(\tau \nabla^\kappa v)(\sigma)|)\, \textup{d}\sigma\lesssim  \Lambda(t) \int_0^t \mathcal{D}^{(\kappa)}[\Psi](\sigma) \, \textup{d}\sigma.
\end{equation}
We refer to \cite{Racke_Said_2019} for the proof when $n=3$. We find by the Cauchy--Schwarz inequality that
 \begin{equation}  \label{R_2_R_4}
 \begin{aligned}
&|\mathcal{F}_1^{(\kappa)}(\nabla^\kappa(v+\tau w))|(\sigma)+|\mathcal{F}_1^{(\kappa)}(\tau \nabla^\kappa v)|\\
\leq&\, \Vert \nabla F^{(\kappa)}\Vert _{L^{2}}\left(%
\Vert\nabla\left( \tilde{v}+\tau \tilde{w}\right)\Vert_{L^2}+\Vert\nabla \tilde{v}\Vert_{L^2}\right).
\end{aligned}
\end{equation}
Keeping in mind how the functional $F^{(\kappa)}$ is defined in \eqref{F_k_Form}, we obtain 
\begin{equation} \label{R_1_k_Nabla}
\begin{aligned}
\Vert \nabla F^{(\kappa)}\Vert _{L^{2}} \lesssim&\, \begin{multlined}[t] \Vert w\Vert
_{L^{\infty }}\Vert \nabla ^{\kappa+1}v\Vert _{L^{2}}+\Vert v\Vert _{L^{\infty
}}\Vert \nabla ^{\kappa+1}w\Vert _{L^{2}} \\
+ \Vert \nabla \psi\Vert _{L^{\infty }}\Vert \nabla ^{\kappa+2}v\Vert
_{L^{2}}+\Vert \nabla v\Vert _{L^{\infty }}\Vert \nabla ^{\kappa+2}\psi\Vert
_{L^{2}}. \end{multlined}
\end{aligned}
\end{equation}%
where we have used inequality \eqref{First_inequaliy_Guass}. We then estimate
\begin{equation}
\begin{aligned}
&\Vert \nabla \psi\Vert _{L^{\infty }}\Vert \nabla ^{\kappa+2}v\Vert _{L^{2}}+\Vert
\nabla v\Vert _{L^{\infty }}\Vert \nabla ^{\kappa+2}\psi\Vert _{L^{2}} \\
\lesssim& \, \left( \Vert \nabla \psi \Vert _{L^{\infty }}\Vert \Delta \nabla
^{\kappa}v\Vert _{L^{2}}+\Vert \nabla v\Vert _{L^{\infty }}\Vert \Delta \nabla
^{\kappa}\psi\Vert _{L^{2}}\right) \\
\lesssim& \, \Vert \nabla \psi\Vert _{L^{\infty }}\Vert \Delta \nabla
^{\kappa}v\Vert _{L^{2}}+\Vert \nabla v\Vert _{L^{\infty }}\left( \Vert \Delta
\nabla ^{\kappa}\left( \psi+\tau v\right) \Vert _{L^{2}}+\Vert \Delta \nabla
^{\kappa}v\Vert _{L^{2}}\right).
\end{aligned}
\end{equation}%
Inserting the above estimates into \eqref{R_1_k_Nabla} yields
\begin{equation}  \label{Nabla_R_1_k}
\Vert \nabla F^{(\kappa)}\Vert _{L^{2}}\lesssim \Lambda[\Psi](t) ( \Vert \nabla
\tilde{v}\Vert _{L^{2}}+\Vert \nabla \tilde{w}\Vert _{L^{2}}+\Vert \Delta \tilde{v}\Vert
_{L^{2}}+\Vert \Delta ( \tilde{\psi}+\tau \tilde{v}) \Vert _{L^{2}}).
\end{equation}
The above bound taken together with estimate \eqref{R_2_R_4} implies \eqref{I_2_4_D_k}. We can estimate $|\mathcal{F}^{(\kappa)}_{0}(\nabla^\kappa w)|$ analogously to arrive at
\begin{equation}\label{I_5_Main_Est}
\int_0^t |\mathcal{F}_0^{(\kappa)}(\nabla^\kappa w)(\sigma)|\, \textup{d}\sigma\lesssim \Lambda[\Psi](t) \int_0^t (\mathcal{D}^{(\kappa-1)}[\Psi](\sigma) \, +  \mathcal{D}^{(\kappa)}[\Psi](\sigma)) \, \textup{d}\sigma.
\end{equation}
Finally, we provide an estimate of $|\mathcal{F}_1^{(\kappa)}(\nabla^{\kappa}(\psi+\tau v))|$. Observe that 
\begin{eqnarray*}
|\mathcal{F}_1^{(\kappa)}(\nabla^{\kappa}(\psi+\tau v))|=|\mathcal{F}_0^{(\kappa)}(\Delta \nabla^{\kappa}(\psi+\tau v))|. 
\end{eqnarray*}
Hence we have
\begin{equation}
|\mathcal{F}_0^{(\kappa)}(\Delta \nabla^{\kappa}(\psi+\tau v))| \leq \Vert F^{(\kappa)}\Vert _{L^{2}} \Vert \Delta   ( \tilde{\psi}+\tau
\tilde{v}) \Vert_{L^2}. 
\end{equation}
Again by keeping in mind how $F^{(\kappa)}$ is defined in \eqref{F_k_Form}, we have
\begin{equation} \label{R_1_k_0_1}
\begin{aligned}
\Vert F^{(\kappa)}\Vert _{L^{2}} \lesssim& \, \begin{multlined}[t] \Vert \lbrack \nabla
^{\kappa},v]w\Vert _{L^{2}}+\Vert v\Vert _{L^{\infty }}\Vert \tilde{w}\Vert
_{L^{2}} \\
+ \Vert \lbrack \nabla ^{\kappa},\nabla \psi]\nabla v\Vert _{L^{2}}+\Vert
\nabla \psi\Vert _{L^{\infty }}\Vert \nabla \tilde{v}\Vert _{L^{2}}. \end{multlined}
\end{aligned}
\end{equation}
By exploiting the commutator estimate from Lemma~\ref{Guass_symbol_lemma}, we find
\begin{equation}
\Vert \lbrack \nabla ^{\kappa},v]w\Vert _{L^{2}}\lesssim (\Vert%
\nabla v\Vert _{L^{\infty }}\Vert \nabla ^{\kappa-1}w\Vert _{L^{2}}+\Vert w\Vert
_{L^{\infty }}\Vert \nabla ^{\kappa}v\Vert _{L^{2}}).
\end{equation}%
Similarly,
\begin{equation}
\begin{aligned}
\Vert \lbrack \nabla ^{\kappa},\nabla \psi]\nabla v\Vert _{L^{2}} \lesssim& \, (\Vert
\nabla v\Vert _{L^{\infty }}\Vert \nabla ^{\kappa+1}\psi\Vert _{L^{2}}+%
\Vert \nabla^2 \psi\Vert _{H^{\frac{n-2}{2}}}\Vert \nabla ^{\kappa}v\Vert
_{H^{1}})  \\
\lesssim& \,\begin{multlined}[t] \Vert \nabla v\Vert _{L^{\infty }}\left( \Vert \nabla ^{\kappa+1}\left(
\psi+\tau v\right) \Vert _{L^{2}}+\Vert \nabla ^{\kappa+1}v\Vert _{L^{2}}\right)\\ +%
\Vert \nabla^2 \psi\Vert _{H^{\frac{n-2}{2}}}(\Vert \nabla ^{\kappa}v\Vert
_{L^{2}}+\Vert \nabla ^{\kappa+1}v\Vert
_{L^{2}}). \end{multlined}
\end{aligned}
\end{equation}
Plugging the last two estimates into \eqref{R_1_k_0_1} results in
\begin{equation} \label{Ineqa_1_R_1}
\begin{aligned}
\Vert F^{(\kappa)}\Vert _{L^{2}} \lesssim& \,\begin{multlined}[t] \Lambda[\Psi](t)\left(
\Vert \nabla ^{\kappa-1}w\Vert _{L^{2}}+\Vert  \nabla^{\kappa+1} v\Vert _{L^{2}}\right.\\
\left.
+\Vert \nabla^{\kappa-1} \nabla v\Vert _{L^{2}}+\Vert \nabla^{\kappa-1}\Delta
\left( \psi+\tau v\right) \Vert _{L^{2}}\right) .  \end{multlined}
\end{aligned}
\end{equation}%
Altogether, we have
\begin{equation}\label{I_3_Estimate_Main}
\int_0^t |\mathcal{F}_1^{(\kappa)}(\nabla^{\kappa}(\psi+\tau v))|\, \textup{d}\sigma\lesssim \Lambda[\Psi](t) \int_0^t (\mathcal{D}^{(\kappa-1)}[\Psi](\sigma) \, +  \mathcal{D}^{(\kappa)}[\Psi](\sigma)) \, \textup{d}\sigma.
\end{equation}
This completes the proof.
\end{proof}
Having estimated the nonlinear terms, we are now ready to prove the final energy bound.
\begin{theorem}
	\label{Proposition_Main} Let $b> \tau c^2>\tau c^2_g$. Let $n \geq 3$ and $m>n/2+1$ be an integer. Then 
	the following estimate holds: 
	\begin{equation}  \label{Estimate_Main}
\|\Psi\|^2_{\mathcal{E}_{m-2}(t)}+\|\Psi\|^2_{\mathcal{D}_{m-2}(t)}\lesssim \|\Psi_0\|^2_{\mathcal{E}_{m-2}(0)}+\|\Psi\|_{\mathcal{E}_{m-2}(t)}\|\Psi\|^2_{\mathcal{D}_{m-2}(t)}.
	\end{equation}
\end{theorem}
\begin{proof}
By summing up \eqref{Energy_k_N_Main_Estimate} over $\kappa=1,\dots, m-2$,  and adding the result to estimate \eqref{Main_Estimate_D_0} ($\kappa=0$), we find
	\begin{equation}  \label{E_I_Est_D_s}
	\|\Psi\|^2_{\mathcal{E}_{m-2}(t)}+\|\Psi\|^2_{\mathcal{D}_{m-2}(t)}
	\lesssim \, \|\Psi_0\|^2_{\mathcal{E}_{m-2}(0)}+(\Lambda[\Psi](t)+\|\Psi\|_{\mathcal{E}_{m-2}(t)}) \|\Psi\|^2_{\mathcal{D}_{m-2}(t)}.
	\end{equation}
 It remains to estimate $\Lambda[\Psi](t)$. Since $m>n/2+1$, we can employ the embeddings
 		\begin{equation}
 \begin{aligned}
 &  \nabla \psi,   \nabla v\in H^{m-1}(\mathbb{R%
 }^n) \hookrightarrow  L^\infty(\mathbb{R%
 }^n),  \\
 & \nabla \psi, \nabla^2 \psi,  \nabla v, \nabla w \in H^{m-2}(\mathbb{R%
 }^n) \hookrightarrow  {H}^{\frac{n-2}{2}}(\R^n),  \\
 & v,  w \in {H}^{m-2}(\mathbb{R%
 }^n) \hookrightarrow  {H}^{\frac{n-2}{2}}(\R^n),
 \end{aligned}
 \end{equation}
to infer that
	\begin{equation}
	\Lambda[\Psi](t)\lesssim \|\Psi\|_{\mathcal{E}_{m-2}(t)}, \quad t \in [0,T].
	\end{equation}
By plugging the above bound into \eqref{E_I_Est_D_s}, we conclude that \eqref{Estimate_Main} holds true.
\end{proof}
\section{Global solvability in the subcritical case} \label{Sec:GlobalExistence}
Equipped with the uniform bound \eqref{Estimate_Main}, we can now prove global solvability of the JMGT equation with memory in $\R^n$, where $n \geq 3$.
\begin{theorem} \label{Thm:GlobalExistence} Let $b> \tau c^2>\tau c^2_g$ and $n\geq 3$. Assume that $\Psi_0 \in \mathcal{H}^{m-1}$ for an integer $m>n/2+1$.  Then there exists a positive
	constant $\delta ,$ such that if
	\begin{equation} 
		\|\Psi_0\|^2_{\mathcal{E}_{m-2}(0)} \leq \delta, 
	\end{equation}
then problem \eqref{Main_System},  \eqref{Main_System_IC} has a global solution
\begin{equation}
\begin{aligned}
\Psi \in \{\Psi=(\psi, v ,w, \eta)^T:&\,  \Psi \in C([0, \infty); \mathcal{H}^{m})\}.
\end{aligned}
\end{equation}
\end{theorem}
\begin{proof}
Let $T>0$ be the maximal time of local existence given by Theorem \ref{Thm:LocalExistence}. We wish to prove by a continuity argument that the norm 
	\begin{eqnarray*}
		|||\Psi |||_{(0,t)}=\| \Psi \|_{\mathcal{E}_{m-2}(t)}+\|\Psi \|_{\mathcal{D}_{m-2}(t)}
	\end{eqnarray*}
	is  uniformly bounded for all time provided that the initial energy is sufficiently small. We have
	\[ \|\Psi\|_{\mathcal{B}_L(0,t)} \leq \| \Psi \|_{\mathcal{E}_{m-2}(t)}+\|\Psi \|_{\mathcal{D}_{m-2}(t)}=|||\Psi |||_{(0,t)},\] where 
	\begin{equation}
	 \|\Psi\|_{\mathcal{B}_L(0,t)}=\sup_{0\leq \sigma \leq t} \|\Phi\|_{\mathcal{H}^{s}}\equiv\|\Phi\|_{\mathcal{E}_{m-2}(t)}
	\end{equation} Thanks to the previous section, we have the energy bound
	\begin{equation}
	\begin{aligned}
	\|\Psi\|^2_{\mathcal{E}_{m-2}(t)}+\|\Psi\|^2_{\mathcal{D}_{m-2}(t)} 
	\lesssim&\,  \|\Psi\|^2_{\mathcal{E}_{m-2}(0)}+ \|\Psi\|_{\mathcal{E}_{m-2}(t)}\|\Psi\|^2_{\mathcal{D}_{m-2}(t)}, \quad t \in [0,T].
	\end{aligned}
	\end{equation}
	Therefore, for all $t\in[0,T]$, we have
	\begin{eqnarray}\label{Main_Y_Estimate}
	|||\Psi|||_{(0,t)}\leq \|\Psi_0\|_{\mathcal{E}_{m-2}(0)}+C |||\Psi|||_{(0,t)}^{3/2},
	\end{eqnarray}
	Thanks to Lemma \ref{Lemma_Stauss}, this means that there exists a positive constant $C$, independent of $t$, such that
	\begin{eqnarray*}
		|||\Psi|||_{(0,t)} \leq C. 
	\end{eqnarray*}
	This uniform bound implies that the local solution can be extended to $T=\infty$. 
\end{proof}
\noindent Accordingly, the JMGT equation in hereditary media with quadratic gradient nonlinearity and initial conditions $\Psi_0 \in \mathcal{H}^{m-1}$ admits a unique solution $\psi$ such that
\begin{equation}
\begin{aligned}
& \nabla \psi \in C([0, +\infty); H^{m-1}(\R^n))\, \cap \, C^1([0, +\infty); H^{m-2}(\R^n)),\\
& \psi_t \in  C([0, +\infty); H^{m}(\R^n)) \, \cap \, C^1([0, +\infty); H^{m-1}(\R^3)), \\
& \psi_{tt} \in C([0, +\infty); H^{m-1}(\R^n))\, \cap  \, C^1([0, +\infty); H^{m-2}(\R^n)), 
\end{aligned}
\end{equation}
where $m >n/2+1$ is an integer, provided that the non-critical condition $b>\tau c^2>\tau c^2_g$ holds. In particular, in a three-dimensional setting, we have $m \geq 3$.
\bibliography{references}{}
\bibliographystyle{siam} 
\end{document}